\documentclass[a4paper,9pt]{article}

\usepackage{amssymb,amsfonts,amsmath,amsthm}
\numberwithin{equation}{section}
\usepackage{mathtools,dsfont}
\usepackage{paralist,enumitem,bm,placeins,url}
\usepackage{cite,array,color}
\usepackage{answers}
\usepackage[margin=30mm,top=20mm]{geometry}
\usepackage{graphicx}%
\usepackage{multirow}%
\usepackage{mathrsfs}%
\usepackage[title]{appendix}%
\usepackage{xcolor}%
\usepackage{textcomp}%
\usepackage{manyfoot}%
\usepackage{booktabs}%
\usepackage{algorithm}%
\usepackage{algorithmicx}%
\usepackage{algpseudocode}%
\usepackage{listings}%
\usepackage[numbers,sort&compress]{natbib}
\usepackage{esint}

\usepackage{tocbasic}
\DeclareTOCStyleEntry[
beforeskip=.2em plus 1pt,
]{tocline}{section}

\definecolor{FrameColor}{rgb}{0.85,0.85,0.85}

\newtheorem{theorem}{Theorem}[section]
\newtheorem{lemma}[theorem]{Lemma}
\newtheorem{proposition}[theorem]{Proposition}

\newtheorem{remark}{Remark}[section]
\newtheorem{definition}{Definition}[section]

\definecolor{LinkColor}{rgb}{0,0,1}
\definecolor{LinkColor2}{rgb}{0,0.5,0}
\definecolor{lg}{rgb}{.5,.5,.5}
\usepackage{hyperref} 
\makeatletter
\hypersetup{%
	colorlinks	=true,%
	linkcolor	=LinkColor,%
	citecolor	=LinkColor2,%
	urlcolor	=LinkColor,%
}
\makeatother

\newcommand{\suchthat}{\;\ifnum\currentgrouptype=16 \middle\fi|\;}

\makeatletter
\renewenvironment{proof}[1][\proofname]{%
	\par\pushQED{\qed}\normalfont%
	\topsep6\p@\@plus6\p@\relax
	\trivlist\item[\hskip\labelsep\bfseries#1\@addpunct{.}]%
	\ignorespaces
}{%
	\popQED\endtrivlist\@endpefalse
}
\makeatother

\makeatletter
\renewcommand\paragraph{\@startsection{paragraph}{4}{\z@}%
	{1ex \@plus1ex \@minus.2ex}%
	{-1em}%
	{\normalfont\normalsize\bfseries}}
\renewcommand\subparagraph{\@startsection{paragraph}{4}{\z@}%
	{1ex \@plus1ex \@minus.2ex}%
	{-1em}%
	{\normalfont\normalsize\itshape}}
\makeatother

\begin{document}
	
	\title{\sc Attractors for Singular-Degenerate Porous Medium Type Equations Arising in Models for Biofilm Growth}
	
	\author{
		{\sc Zehra Sen}\thanks{Email: \href{mailto:zarat@hacettepe.edu.tr}{zarat@hacettepe.edu.tr}}\\
		\small  Department of Mathematics, Faculty of Science,
		Hacettepe University, \\
		\small   Beytepe 06800, Ankara, Türkiye.
		\\[2ex]	
		{\sc Stefanie Sonner}\thanks{Email: \href{mailto:stefanie.sonner@ru.nl}{stefanie.sonner@ru.nl}}  \\
		\small IMAPP-Mathematics, PO Box 9010, Radboud University, \\
		\small  6500 GL, Nijmegen, The Netherlands.
		\smallskip}	
	
	\date{}
	\maketitle
\begin{abstract}
We investigate the long-time behaviour of solutions of a class of singular-degenerate porous medium type equations in bounded domains with homogeneous Dirichlet boundary conditions. The existence of global attractors is shown under very general assumptions. Assuming, in addition, that solutions are globally Hölder continuous and the reaction terms satisfy a suitable sign condition in the vicinity of the degeneracy, we also prove the existence of an exponential attractor, which, in turn, yields the finite fractal dimension of the global attractor. Moreover, we extend the results for scalar equations to systems where the degenerate equation is coupled to a semilinear reaction-diffusion equation. The study of such systems is motivated by models for biofilm growth.
\end{abstract}
\begin{small}
	\smallskip
	\noindent{{\bf MSC Classification:}  
		37L30, 35K65, 35K67, 35K10, 92D25

	}
\end{small}

\vfill

\begin{small}
	\setcounter{tocdepth}{2}
	\hypersetup{linkcolor=black}
	
\end{small}

\newpage
\setlength\parindent{0ex}
\setlength\parskip{1ex}

\section{Introduction}\label{sec1}

\subsection{Singular-degenerate porous medium type equations}
We investigate the long-time behaviour, in terms of global and exponential attractors, of a class of second-order quasilinear degenerate parabolic equations considered in \cite{MullerSonner}. The problems are of the form 
\begin{align}
	u_{t}=\Delta \phi(u)+f(\cdot,u)\ \ \text{in}\ \ (0,T]\times\Omega,  \label{eq:SE}
\end{align}
complemented with a suitable initial data and homogeneous Dirichlet boundary conditions, where $T>0$ and $\Omega\subset \mathbb{R}^{n}, n\geq 1$, is a bounded domain. The solution $u$ assumes its values in the interval $(-1,1)$, $f:\Omega\times (-1,1)\to \mathbb{R}$ is continuous and $\phi: (-1,1) \to \mathbb{R}$ is strictly increasing with a degeneracy $\phi'(0)=0$ and singularities $\phi(\pm 1)=\pm\infty$. Since the degeneracy in $u=0$ is of the same type as in the porous medium equation, $u_{t}=\Delta u^{m}, m>1$, 
we call \eqref{eq:SE} a singular-degenerate equation of porous medium type. Porous medium type equations are used to model various natural phenomena such as the flow of an ideal gas through a porous medium (\cite{Muskat}), nonlinear heat transfer 
(\cite{Zeldovich}) and the spreading of certain biological populations (\cite{GurtinMacCamy}). 

We will also consider the following system, where the equation \eqref{eq:SE} is coupled to a semilinear reaction-diffusion equation,
\begin{equation}
	\left\{ 
	\begin{array}{l}
		u_{t}=\Delta \phi(u)+f(\cdot,u,v), \\
		v_{t}=\Delta v+g(\cdot,u,v),
	\end{array}
	\right.\ \ \text{in}\ (0,T]\times\Omega, \label{eq:CS}
\end{equation}
equipped with suitable initial values and homogeneous Dirichlet boundary data. Such systems arise in models for the growth of bacterial biofilms. 
Biofilms are dense aggregations of microorganisms that form in moist environments. In a biofilm, cells are encased in a slimy matrix of extracellular polymeric substances and adhere to one another and usually to surfaces. A reaction-diffusion model that predicts heterogeneous spatial structures of fully grown biofilms on the mesoscale was introduced and numerically studied by Eberl et al. in \cite{Eberl}. 
It is a system of a quasilinear reaction-diffusion equation for the biomass density $M$ and a semilinear reaction-diffusion equation for the nutrient concentration $C$. The equations are coupled via Monod reaction functions modeling biomass production and nutrient consumption. For the dimensionless variables $C$ and $M$, which are normalized, respectively, with respect to the bulk concentration and the maximum biomass density, the model reads as follows:
\begin{equation}
	\left\{ 
	\begin{array}{l}
		\partial_{t} M=d_{2}\nabla\cdot(D(M)\nabla M)-K_{2}M+K_{3}\frac{CM}{K_{4}+C}, \\
		\partial_{t} C=d_{1}\Delta C-K_{1}\frac{CM}{K_{4}+C}
	\end{array}
	\right. \ \text{in}\ (0,T]\times\Omega, \label{eq:BM}
\end{equation}
where $\Omega\subset \mathbb{R}^{n}, n = 1, 2, 3$, is a bounded domain and $d_{1}, d_{2}>0$, $K_{1}, K_{2}, K_{3}\geq 0$, and $K_{4}> 0$ are given constants. The subregion\begin{align*}
	\Omega_{M}(t)=\{x\in \Omega:M(t,x)>0\}
\end{align*}
defines the actual biofilm situated on the surface, which typically corresponds to the bottom part of the boundary $\partial \Omega$. The biofilm diffusion coefficient $D(M)$ is given by \begin{equation*}
	D(M)=\frac{M^{b}}{(1-M)^{a}},\quad a\geq 1, b>0,
\end{equation*}
which has a degeneracy of porous medium type at $M=0$. It ensures a finite speed of interface propagation and a sharp interface between the biofilm and the surrounding liquid. On the other hand, the singularity at $M = 1$ implies that the spatial spreading becomes very large whenever $M$ approaches values close to $1$. As a consequence, the biomass density remains bounded by its maximum value despite the growth terms in the equation. It is important to note that by setting 
\begin{equation*}
	\phi(u)=\int\limits_{0}^{u}\frac{z^{b}}{(1-z)^{a}}dz,
\end{equation*}we observe that the biofilm-growth model \eqref{eq:BM} is a particular case of the system
\eqref{eq:CS} that we consider. Moreover, in simulation studies, $\Omega$ is typically rectangular and mixed Dirichlet-Neumann boundary conditions are imposed. More specifically, 
homogeneous Neumann boundary conditions for $C$ and $M$ are assumed on the bottom part of the boundary, due to the impermeability of the surface to nutrients and biomass, and also on the lateral boundaries. Nutrients are added through the top boundary, which is described by Dirichlet boundary conditions for $C$, while homogeneous Dirichlet boundary conditions are imposed for $M$. Initially, there are small pockets of biomass of density $M_{0}< 1$  
on the bottom part of the boundary, resembling the initial biofilm colonies on the surface, while $M_0$ vanishes everywhere else in the domain. The nutrient concentration $C_{0}$ is set to the bulk concentration everywhere. The well-posedness of scalar equations and coupled systems of the form \eqref{eq:SE} and \eqref{eq:CS} with  mixed Dirichlet-Neumann boundary conditions on bounded Lipschitz domains has been established in \cite{MullerSonner}. In this paper, to prove the existence of global attractors we consider smooth non-degenerate approximations and therefore consider Problems \eqref{eq:SE} and \eqref{eq:CS} with homogeneous Dirichlet boundary conditions. 

\subsection{Global and exponential attractors: previous results}
The long-time behaviour of a large class of dissipative evolution PDEs can be investigated in terms of global attractors. Global attractors are compact invariant sets of the phase space that attract all bounded subsets under the temporal evolution of the generated semigroup. The global attractor is unique and the minimal compact attracting set. Moreover, for many problems on bounded domains, the (Hausdorff and fractal) dimension of the global attractor is finite (see, e.g., \cite{Hale}, \cite{Temam}, \cite{BabinVishik}, \cite{ChepyzhovVishik}, \cite{Prazak}, \cite{Efendiev2010}). However, global attractors have some drawbacks. For instance, the rate of attraction can be very slow (\cite{Kostin}), global attractors can be sensitive to perturbations (\cite{Raugel}), and in some situations, they cannot capture important transient behaviours (\cite{Takei}). To overcome these drawbacks, Eden et al. \cite{Eden} suggested to consider larger objects, namely, exponential attractors, which are compact positively invariant sets that contain the global attractor, have finite fractal dimension, and attract all bounded sets at an exponential rate. Exponential attractors are more robust under perturbations than global attractors due to the exponential rate of attraction (see, for instance, \cite{MiranvilleZelik2004}). Moreover, the existence of an exponential attractor implies the existence and finite fractal dimension of the global attractor. On the other hand, exponential attractors are only positively invariant and therefore not unique. Hence, there are different methods to construct an exponential attractor.

The first existence proof for exponential attractors in \cite{Eden} was non-constructive, based on the squeezing property of the semigroup and restricted to a Hilbert space setting. Efendiev et al. \cite{EfendievMiranvilleZelik2000} provided an alternative, explicit construction of exponential attractors for semigroups in Banach spaces based on the decomposition the semigroup into a compact and contracting part. In \cite{CzajaEfendiev}, Czaja and Efendiev extended this construction and the existence proof for exponential attractors to a setting that can be applied to degenerate parabolic equations. It covers a noteworthy smoothing estimate formulated in suitable function spaces related to pieces of trajectories of the semigroup and not in the phase space. Efendiev and Zelik further generalized this approach allowing for estimates and function spaces that depend on points in a compact, positively invariant set to construct exponential attractors for equations of porous medium type, see \cite[Remark 3.2]{EfendievZelik2008}. For an overview and comparison of commonly used construction methods for exponential attractors we refer to \cite{CzajaSonner}.

It is important to note that most results on global and exponential attractors concern non-degenerate problems. However, the situation is significantly more complex for equations with singularities and/or degeneracies, such as for porous medium type equations and elliptic-parabolic problems. 
Among the vast literature on the well-posedness and regularity of solutions of degenerate diffusion equations, we particularly mention \cite{Efendiev2009, MitraSonner,Muller} and \cite{MullerSonner} which specifically address systems of the form \eqref{eq:CS}. In \cite{Efendiev2009}, Efendiev et al. developed a well-posedness theory for the biofilm growth model \eqref{eq:BM} in a smooth bounded domain with homogeneous Dirichlet boundary conditions for $M$ based on smooth non-degenerate approximations. However, the analysis does not cover more general domains, non-vanishing Dirichlet boundary data for $M$ and mixed Dirichlet-Neumann boundary conditions. To overcome these restrictions, Hissink Muller and Sonner \cite{MullerSonner} provided  well-posedness results for a broader class of quasilinear parabolic problems, including our problem \eqref{eq:SE}. In particular, the existence, uniqueness and continuous dependence on initial data were shown. The well-posedness result was also extended for coupled systems, including our problem \eqref{eq:CS}, by using Banach's fixed point theorem. In \cite{MitraSonner}, the well-posedness theory was further generalized for systems involving couplings to ODEs.  Furthermore, Hissink Muller \cite{Muller} proved the interior Hölder continuity of local weak solutions of \eqref{eq:SE} that are bounded away from $1$ by intrinsic scaling methods. 

As for the existence of global attractors for porous medium type equations, we mention \cite{Feireisl}, \cite{Andreu}, \cite{Carvalho} and \cite{Efendiev2009}. In particular, Feireisl et al.  considered in \cite{Feireisl} an equation of the form\begin{align}
	u_{t}-\Delta\phi(u)+f(u)=g, \label{eq:Feireisl}
\end{align}in $(0,\infty)\times\mathbb{R}^{n}$ and proved, by applying the splitting method, that the generated semigroup  possesses a global attractor in the weighted space $L^{1}(\mathbb{R}^{n};\rho(x)dx)$ where the weight $\rho \in L^{1}(\mathbb{R}^{n})$ goes to zero at infinity. To prove the asymptotic compactness of the semigroup for our problem we use arguments from this paper. Andreu et al. \cite{Andreu} proved the existence of a global attractor in $L^{\infty}(\Omega)$ for \eqref{eq:Feireisl} with $g=0$ in a bounded smooth domain $\Omega\subset \mathbb{R}^{n}$ with nonlinear Neumann boundary conditions under stronger assumptions on the nonlinearities $\phi$ and $f$. Moreover, Carvalho et al. \cite{Carvalho} considered \eqref{eq:Feireisl} with  $\phi(u)=|u|^{m}$, $m>1$, and $g=0$ in
bounded smooth domains with homogeneous Dirichlet boundary conditions and a globally Lipschitz continuous function $f$ and proved the existence of a global attractor in $H^{-1}(\Omega)$ based on the theory of maximal monotone operators. Finally, in the aforementioned paper \cite{Efendiev2009}, the authors showed that the biofilm growth model \eqref{eq:BM} in a smooth bounded domain $\Omega$ generates a semigroup  which  is Lipschitz continuous in $L^{1}(\Omega)\times L^{1}(\Omega)$ that possesses a global attractor.

To the best of our knowledge, only few results have been obtained concerning the finite fractal dimension of global attractors for equations with singular and/or degenerate diffusion. The primary challenge for degenerate problems is the lack of regularity (smoothing effect) in the vicinity of the degeneracies, which makes standard techniques inapplicable. For instance, the unstable manifold approach (see \cite{Temam}) cannot be applied to derive lower bounds for the dimension of the attractor because the semigroups 
are typically not differentiable. The finite dimensionality of global attractors for porous medium type equations has been shown in \cite{Eden1991}, \cite{EfendievZelik2008} and \cite{EfendievZhigun}. In particular, Eden et al. \cite{Eden1991} proved that doubly nonlinear porous medium type equations of the form\begin{align*}
	\frac{\partial\beta(u)}{\partial t}-\Delta u+g(x,u)=0,
\end{align*}in bounded domains $\Omega$ with homogeneous Dirichlet boundary conditions possess a global attractor in $L^{2}(\Omega)$. Also, assuming higher regularity of the nonlinearities $\beta$ and $g$, the authors showed the finite fractal dimension of the global attractor. In \cite{EfendievZelik2008}, Efendiev and Zelik showed the existence of global and exponential attractors for \eqref{eq:Feireisl} with  $\phi(u)=|u|^{m}$ in bounded domains with homogeneous Dirichlet boundary conditions. 
Assuming that the derivative of the reaction term $f$ is positive at zero, first, the finite fractal dimension of the global attractor was shown, and this result was subsequently used to construct an exponential attractor. Moreover, based on upper and lower bounds for the Kolmogorov $\varepsilon$-entropy the authors provided sufficient conditions on the nonlinearity $f$ that lead to a global attractor with infinite fractal dimension. Using the approach in \cite{EfendievZelik2008}, Efendiev and Zhigun proved in \cite{EfendievZhigun} the existence of an exponential attractor for a class of parabolic systems with degenerate diffusion and chemotaxis assuming suitable conditions on the order of the degeneracy and growth of the chemotactic function.

\subsection{Aim and outline}
Our aim is to prove the existence of global and exponential attractors for problems \eqref{eq:SE} and \eqref{eq:CS} under suitable structural assumptions. The existence of exponential attractors ensures the existence and finite fractal dimension of the global attractor. However, we state a separate existence theorem for global attractors as it holds under much weaker assumptions. It generalizes the results in \cite{Efendiev2009} on the existence of a global attractor for the biofilm model \eqref{eq:BM}.
To construct exponential attractors 
we additionally need a sign condition on $f$ at zero and the global H\"older continuity of solutions, similarly as in \cite{EfendievZelik2008}. 
We simplify the latter proof, where first the finite fractal dimension of the global attractor was shown and then the exponential attractor was constructed,  
by using an abstract existence theorem for exponential attractors, see \cite[Remark 3.2]{EfendievZelik2008} and \cite{CzajaEfendiev}. 
The finite fractal dimension of the global attractor is then  a consequence of the existence of an exponential attractor. Our results are formulated for a broad class of degenerate problems in bounded domains with homogeneous Dirichlet boundary conditions. 
To the best of our knowledge, they are novel in the context of global and exponential attractors.
We remark that system \eqref{eq:CS} is not restricted to biofilm models. Such systems can also be used to describe other phenomena with sharp interfaces propagating at a finite speed. Moreover, our results easily extend to, e.g., the model for quorum sensing induced detachment in biofilms \cite{Emerenini} which was analytically studied in \cite{EmereniniSonner}. It is an extension of the biofilm growth model \eqref{eq:BM} involving one singular-degenerate equation and three semi-linear
reaction diffusion equations for dissolved substrates.

We organize the paper as follows: In Section \ref{sec:HMR}, we introduce the functional setting and state the main results on the existence of global and exponential attractors. In Section \ref{sec:GA}, we first prove the existence of a global attractor for the scalar Problem \eqref{eq:SE}. To this end, we consider smooth nondegenerate approximations for Problem \eqref{eq:SE} that possess smooth classical solutions. We derive uniform bounds and smoothing estimates for the approximate solutions and passing to the limit we obtain a solution of the degenerate problem. This allows us to prove that the original problem \eqref{eq:SE} generates a Lipschitz continuous semigroup in a suitable phase space endowed with the $L^{1}(\Omega)$-topology. We show the existence of a bounded absorbing set and the asymptotic compactness of the trajectories, and then extend the result to the coupled system \eqref{eq:CS}. In Section \ref{sec:EA}, under the key assumptions of global Hölder continuity of solutions  and appropriate sign conditions on the partial derivative of the nonlinearity $f$ in the vicinity of the degeneracy, we prove the existence of exponential attractors for \eqref{eq:SE}. As a consequence, the global attractor has finite fractal dimension. We also extend the existence result for exponential attractors for the coupled system \eqref{eq:CS}.

\section{Hypotheses and Main Results}\label{sec:HMR}
In this section, we state the hypotheses and main results. The proofs will be given in  subsequent sections. 

\subsection{Well-Posedness Results}

We recall the well-posedness results in \cite{MullerSonner} and \cite{MullerThesis} for the scalar equation \eqref{eq:SE} and the coupled system \eqref{eq:CS}, and introduce the functional setting for the associated semigroups. 
First, we consider the following initial-boundary value problem for \eqref{eq:SE} with homogeneous Dirichlet boundary conditions,
\begin{equation}
	\begin{cases}
		u_{t}=\Delta \phi(u)+f(\cdot,u)&\text{in}\ (0,T]\times\Omega,\\
		u=u_{0} & \text{in}\ \{0\}\times\Omega,\\
		\phi(u)=0& \text{on}\ (0,T]\times \partial\Omega,
	\end{cases}
	 \label{eq:SP}
\end{equation} 
where $T>0$, $\Omega\subset \mathbb{R}^{n}, n\geq 1,$ is a bounded domain with boundary $\partial\Omega$. Let $\Omega_{T}=(0,T]\times\Omega$ and $I=(-1,1)$. 

We assume that  $\phi:I\to \mathbb{R}$ satisfies the following structural assumptions,
\begin{align}
	&\phi\in C(I),\ \phi(0)=0,\ \phi\ \text{is strictly increasing},\label{eq:P1}\\ &\phi\ \text{is surjective,} \label{eq:P2}\\ 
	&\left\{
	\begin{array}{ll}
		\phi \ \text{is piece-wise continuously differentiable, } \phi'(0)=0\\
		\phi \ \text{is\ convex on}\ [0,1)\ \text{and}\ \text{concave on}\ (-1,0]. 
	\end{array}  \right.\label{eq:P3}
\end{align}

\begin{remark}
	\begin{itemize}
		\item[i)] The singularities of $\phi$ are encoded in \eqref{eq:P2}, which implies that $\phi(\pm 1)=\pm \infty$. 
		\item[ii)] The condition $\phi'(0)=0$ in \eqref{eq:P3} reflects that $\phi$ has a porous medium type degeneracy. 
		\item[iii)] The assumptions \eqref{eq:P1} and \eqref{eq:P2} are sufficient for the well-posedness. The condition \eqref{eq:P3} yields additional regularity of the solutions (see \cite{MullerSonner}).
	\end{itemize}
\end{remark}
We suppose that the reaction function $f:\Omega\times I \rightarrow\mathbb{R}$ satisfies the \textit{Carathéodory conditions}, i.e., $f(\cdot,z)$ is measurable for all $z\in I$,  $f(x,\cdot)$ is continuous for almost every $x\in \Omega$, and the function $x\to \lVert f(x,\cdot)\rVert_{C(J)}$ is summable for each compact $J\subset I$. We further assume that  $f(\cdot,0)$ is essentially bounded and $f$ is uniformly Lipschitz continuous with respect to the second argument, i.e. there exists $L\geq 0$ such that
\begin{align}
	\lVert f(\cdot,z_{1})-f(\cdot,z_{2})\rVert_{L^{\infty}(\Omega)} \leq L|z_{1}-z_{2}|\qquad \forall z_{1}, z_{2}\in I. \label{eq:F1}
\end{align}

\begin{remark}
	Condition \eqref{eq:F1} is satisfied, e.g. if $f:\Omega\times\mathbb{R}\to\mathbb{R}$ is locally Lipschitz continuous with respect to $z\in \mathbb{R}$. 
\end{remark}

We assume that  the initial data 
\begin{equation}
	u_{0}:\Omega\to I\ \text{is measurable and}\ \Phi(u_{0})\in L^{1}(\Omega), \label{eq:IC}
\end{equation}where $\Phi(z):=\int\limits_{0}^{z}\phi(s)ds$. We note that the quantity $\int\limits_{\Omega}\Phi(u)$ represents the (absolute) energy of the solution $u$. 

Now, we introduce the functional setting for problem \eqref{eq:SP}. We consider the closed subspace $ V=H^{1}_0(\Omega)$
of $H^{1}(\Omega)$ and denote its dual by $V^{*}=H^{-1}(\Omega)$. The solutions need to satisfy $\phi(u)\in L^{2}(0,T;V)$ to ensure that $\phi(u)=0$ a.e. on $\partial\Omega$.

\begin{remark}
	It can be inferred from \eqref{eq:P1} and \eqref{eq:IC} that $u_{0}\in L^{1}(\Omega)$ (cf. \cite{MullerSonner}, Remark 2.2). Indeed, it follows from \eqref{eq:P1} that \begin{equation*}
		\Phi(z_{1})-\Phi(z_{2})\geq\phi(z_{2})(z_{1}-z_{2})\ \ \ \forall z_{1}, z_{2}\in I.
	\end{equation*}
	Taking $z\in I, \delta>0$ and setting $z_{1}=z, z_{2}=\text{sgn}(z)\beta(\frac{1}{\delta})$ for $\beta=\phi^{-1}$, and using that $\phi(0)=0$, this implies that  
	\begin{equation*}
		\Phi(z)\geq \Phi(z)-\Phi(z_{2})\geq \frac{1}{\delta}\text{sgn}(z)(z-z_{2})=\frac{1}{\delta}\left(|z|-\beta(\frac{1}{\delta})\right).
	\end{equation*}
	Hence, we obtain
	\begin{equation*}		|z|\leq\delta\Phi(z)+\beta(\frac{1}{\delta}),\quad z\in I,\delta>0, \label{eq:AV}
	\end{equation*} 
	which, together with \eqref{eq:IC}, yields that $u_{0}\in L^{1}(\Omega)$. 
\end{remark}

We denote by $(\cdot,\cdot)$  the dual pairing between $V^{*}$ and $V$ and by $\left<\cdot,\cdot\right>$ the inner product in $L^{2}(\Omega)$, and define a weak solution to problem \eqref{eq:SP} as follows:

\begin{definition}\label{WS}
	A measurable function $u:\Omega_{T}\to I$ is a weak solution of \eqref{eq:SP} if $u\in H^{1}(0,T;V^{*})$ with weak derivative $u_{t}$, $u(0)=u_{0}\in V^*$ and $\phi(u)\in L^{2}(0,T;V)$ that satisfies the scalar equation $\eqref{eq:SP}_{1}$ in  distributional sense, i.e. the identity 
	\begin{equation}
		\int\limits_{0}^{T}[(u_{t},\eta)+\langle\nabla\phi(u),\nabla\eta\rangle]dt=\int\limits_{0}^{T}\langle f(\cdot,u),\eta\rangle dt\label{eq:DI}  
	\end{equation}
	holds for all $\eta\in L^{2}(0,T;V)$. 
\end{definition}

\begin{remark}
	Property \eqref{eq:F1}, the essential boundedness of $f(\cdot,0)$ and the fact that $u\in (-1,1)$ imply that $f(\cdot,u)\in  L^{\infty}(0,T;L^{2}(\Omega))$ so that the right-hand side of \eqref{eq:DI} is well-defined.  
\end{remark}
We formulate the following well-posedness result for the scalar problem \eqref{eq:SP}, which follows from [\cite{MullerSonner}, Theorems 2.6–2.7] in the special case of homogeneous Dirichlet boundary conditions.

\begin{theorem}\label{WPS}
	Under the assumptions \eqref{eq:P1}-\eqref{eq:P2} and \eqref{eq:F1}-\eqref{eq:IC}, problem \eqref{eq:SP} has a unique weak solution $u:\Omega_{T}\to I$ that satisfies the energy estimate 
	\begin{align}
    \begin{split}
		&\lVert\Phi(u)\rVert_{L^{\infty}(0,T;L^{1}(\Omega))}+\lVert\nabla\phi(u)\rVert_{L^{2}(\Omega_{T})}^{2} \\
&		\leq C\left(\lVert\Phi(u_{0})\rVert_{L^{1}(\Omega)}+\lVert f(\cdot,u)\rVert_{L^{2}(\Omega_{T})}\lVert\phi(u)\rVert_{L^{2}(\Omega_{T})}\right), \label{eq:REE}
        \end{split}
	\end{align}
    where $C\geq 0$ is some constant. 
	
	The solution depends continuously on initial data. More specifically, if $u$ and $\Tilde{u}$ are two solutions of \eqref{eq:SP} corresponding to the functions $f$, $\Tilde{f}$ satisfying \eqref{eq:F1} and initial data $u_{0}$, $\Tilde{u_{0}}$ satisfying \eqref{eq:IC}, then the following $L^{1}$-contraction holds,
	\begin{equation}
		\lVert u(t)-\Tilde{u}(t)\rVert_{L^{1}(\Omega)}\leq e^{Lt}\left(\lVert u_{0}-\Tilde{u_{0}}\rVert_{L^{1}(\Omega)}+\int\limits_{0}^{t}\lVert f(\cdot,u)-\Tilde{f}(\cdot,u)\rVert_{L^{1}(\Omega)}dt\right), \label{eq:L1C}
	\end{equation}
	for all $t\in [0, T]$, where $L$ is the Lipschitz constant in \eqref{eq:F1}.
	
	Moreover, if $\phi(u_{0})\in L^{\infty}(\Omega)$, then $\phi(u)$ is bounded by a constant depending on $\Omega$, $L$, $\lVert f(\cdot,0)\rVert_{L^{\infty}(\Omega)}$ and $\phi(u_{0})$. If the condition \eqref{eq:P3} holds, then we have $u\in C([0,T];L^{1}(\Omega))$. If both assumptions hold and $\phi(u_{0})\in V$, then $\phi(u)\in H^{1}(0,T;L^{2}(\Omega))$.
\end{theorem}

Next, we consider the initial-boundary value problem for the coupled system \eqref{eq:CS}. As in the scalar case, we impose homogeneous Dirichlet boundary conditions,
\begin{equation}
	\begin{cases}
		u_{t}=\Delta \phi(u)+f(\cdot,u,v)& \text{in}\ (0,T]\times\Omega, \\
		v_{t}=\Delta v+g(\cdot,u,v)&
		\text{in}\ (0,T]\times\Omega,\\
		u=u_{0}, \ \ v=v_{0} & \text{in}\ \  \{0\}\times\Omega,\\
		\phi(u)=0, v=0&\text{on}\ (0,T]\times\partial\Omega,
        \end{cases}
	 \label{eq:CP}
\end{equation}
and assume that $I=[0, 1)$. This is motivated by applications such as the biofilm growth model \eqref{eq:BM}, where solutions describe nonnegative quantities such as densities or concentrations. We suppose that the reaction functions $f, g :\Omega\times[0, 1)\times[0, 1] \to \mathbb{R}$ are measurable, bounded and uniformly Lipschitz continuous with respect to the last two arguments, i.e. there exists $L\geq 0$ such that
\begin{align}
	&\quad \lVert f(\cdot,u_{1},v_{1})-f(\cdot,u_{2},v_{2})\lVert_{L^{\infty}(\Omega)}+\lVert g(\cdot,u_{1},v_{1})-g(\cdot,u_{2},v_{2})\lVert_{L^{\infty}(\Omega)}\notag\\
	&\leq L\left(|u_{1}-u_{2}|+|v_{1}-v_{2}|\right)\qquad \forall u_{1},u_{2}\in [0,1), v_{1},v_{2}\in [0,1]. \label{eq:FG1}  
\end{align}
Moreover, to ensure that the solutions remain nonnegative and bounded, we assume that
\begin{equation}
	f(\cdot,0, v)\geq 0\ ,\ \ 
	g(\cdot,u, 0)\geq 0\ \text{ and }\ g(\cdot,u,1)\leq 1\ \ \text{in}\ \Omega\quad  \forall u\in [0,1),  v\in [0,1]. \label{eq:FG2}
\end{equation}

\begin{remark}
	The condition on $f$ in \eqref{eq:FG2} yields the nonnegativity of the solution $u$, which validates the choice of the interval $I=[0,1)$. It is not an open interval as assumed previously. However, the structural functions can be appropriately extended to the open interval $(-1,1)$. We then apply the well-posedness theory for \eqref{eq:SP} in combination with a comparison principle, which yields that solutions assume their values in $[0, 1)$.  Moreover, the conditions on $g$ in \eqref{eq:FG2} ensure that $v$ takes values in the interval $[0,1]$. For further details, we refer to \cite{MullerSonner}.
\end{remark}

Solutions of system \eqref{eq:CP} are analogously defined as in Definition~\ref{WS} for problem \eqref{eq:SP}.
We recall the well-posedness result for the coupled system \eqref{eq:CP}, given in \cite[Theorem 2.9]{MullerSonner}, which we formulate for the specific case of homogeneous Dirichlet boundary conditions.

\begin{theorem}\label{WPC}
	Assume that the function $\phi:[0,1)\to [0,\infty)$ satisfies the conditions \eqref{eq:P1}-\eqref{eq:P2} and $f, g$ satisfy  \eqref{eq:FG1}-\eqref{eq:FG2}. Moreover, the initial data $u_{0}:\Omega\to [0,1)$ satisfies \eqref{eq:IC} and $v_{0}:\Omega\to [0,1]$ is measurable. Then, there exists a unique solution $(u,v):\Omega_{T}\to[0,1)\times[0,1]$ of problem \eqref{eq:CP} in $L^{\infty}(0,T;L^{1}(\Omega))\times C([0,T];L^{1}(\Omega))$. The solution satisfies an $L^{1}$-contraction estimate and an energy inequality analogous to \eqref{eq:REE} and \eqref{eq:L1C}, respectively. Furthermore, if $\phi$ satisfies \eqref{eq:P3}, then $(u,v)\in C([0,T];L^{1}(\Omega)\times L^{1}(\Omega))$.
	
	In addition, if $u_{0}\leq 1-\theta$ for some $\theta\in(0,1)$, then $u\leq 1-\mu$ for some $\mu\in(0,1)$ depending on $\Omega, L$ and $\phi(1-\theta)$. In this case, if $\phi(u_{0})\in V$ and $\phi$ satisfies \eqref{eq:P3}, then $\phi(u)\in H^{1}(0,T;L^{2}(\Omega))$.
\end{theorem}\begin{remark}
	In view of applications, the last property in Theorem~\ref{WPC} is crucial. Indeed, the biomass density remains
	bounded by a constant strictly less than one due to the singularity
	in the diffusion coefficient $\phi'$.
\end{remark}

\subsection{Long-time Behaviour of Solutions}

In this section we present the main results that address the long-time behaviour of solutions in terms of global and exponential attractors for the semigroups generated by the scalar problem \eqref{eq:SP} and the coupled problem \eqref{eq:CP}. The existence of an exponential attractor will also imply that the fractal dimension of the global attractor is finite.


We recall the definition of a global attractor for a semigroup.
\begin{definition}\label{GA} 
	The set $\mathcal{A}\subset E$ is a global attractor for a semigroup $\left\{ S\left( t\right) \right\}
	_{t\geq 0}$ in a Banach space $E$ iff\par
	\begin{itemize}
		\item[\textbf{i)}] $\mathcal{A}$ is compact in $E$,\par
		\item[\textbf{ii)}] $\mathcal{A}$ is invariant, i.e., $S\left( t\right)
		\mathcal{A}=\mathcal{A}$ $\forall t\geq 0$,\par
		\item[\textbf{iii)}] $\mathcal{A}$ attracts all bounded subsets of $E$, i.e. for any bounded set $D\subset E$ \begin{equation*}
			\lim_{t\rightarrow \infty }d_{E}(S\left(t\right) D|\mathcal{A})=0
		\end{equation*}
		where $d_{E}(A|B)=\sup\limits_{x\in A}dist(x,B)=\sup\limits_{x\in A}\inf\limits_{y\in B}\lVert x-y\rVert$ is the Hausdorff semi-distance.
	\end{itemize}
\end{definition}

We introduce the phase space 
\begin{equation*}
	\mathcal{X}:=\overline{\{u\in L^{1}(\Omega):\|u\|_{L^\infty(\Omega)}<1, \phi(u)\in V\}},
\end{equation*}
where the bar denotes the closure taken in $L^1(\Omega)$. Then, using Theorem~\ref{WPS} we will show that problem \eqref{eq:SP} generates a semigroup $\{\mathcal{S}(t)\}_{t\geq 0}$ in $\mathcal{X}$ that is Lipschitz continuous in $L^{1}(\Omega)$.  More specifically, $\mathcal{S}(t)u_{0}=u(t)$, where $u\in C([0,T];L^{1}(\Omega))$ is the unique weak solution
of \eqref{eq:SP} with initial data $u_{0}\in\mathcal{X}$.
We note that $\mathcal{X}$
is positively invariant and bounded in $L^{1}(\Omega)$.  

We can now state our result on the existence of the global attractor for the problem \eqref{eq:SP}. In addition to \eqref{eq:P1}-\eqref{eq:P3}, we assume that $\phi$ satisfies\begin{align}
    |\phi(z)|\leq c_{1}(1-|z|)^{1-a}+c_{2} \qquad \forall z\in I, \label{eq:adphi}
\end{align}for some constants $a>1$ and $c_1,c_2\geq 0$.

\begin{theorem}\label{GAS}
	Let  \eqref{eq:P1}--\eqref{eq:IC} and \eqref{eq:adphi} be satisfied.
	Then, the semigroup $\{\mathcal{S}(t)\}_{t\geq 0}$ generated by problem \eqref{eq:SP} has a global attractor $\mathcal{A}$ in $\mathcal{X}$ endowed with the $L^{1}(\Omega)$-topology, which has the form $\mathcal{A}=\omega(\mathcal{X})$. 
\end{theorem}
In this theorem, 
\begin{align*} 	\omega(D)=\bigcap_{\tau\geq 0}\overline{\bigcup_{t\geq \tau}S(t)D}
\end{align*}denotes the $\omega$-limit set of a set $D$ for a semigroup $\{S(t)\}_{t\geq 0}$ in a Banach space $E$, where the bar denotes the closure in $E$.

Analogously, we define the phase space for the coupled problem 

\begin{equation*}
	\mathcal{Y}:=\overline{\{(u,v)\in L^{1}(\Omega)\times L^{1}(\Omega):\|u\|_{L^\infty(\Omega)}<1, \phi(u)\in V,u\geq 0, v\in [0,1] \text{ measurable}\}},
\end{equation*}
where the bar denotes the closure taken in $L^1(\Omega)\times L^{1}(\Omega)$. 
Then, using Theorem~\ref{WPC} we will show that the coupled problem \eqref{eq:CP} generates a  semigroup $\{\mathcal{T}(t)\}_{t\geq 0}$ in $\mathcal{Y}$ that is Lipschitz continuous in $L^{1}(\Omega)\times L^{1}(\Omega)$. More specifically,  $\mathcal{T}(t)(u_{0},v_{0})=(u(t),v(t))$, where $(u,v)\in C([0,T];L^{1}(\Omega))\times C([0,T];L^{1}(\Omega))$ is the unique weak solution
of \eqref{eq:CP} with initial data $(u_{0},v_{0}) \in \mathcal{Y}$.
Note that $\mathcal{Y}$ is  positively invariant and bounded in $L^{1}(\Omega)\times L^{1}(\Omega)$. 

We now state our result on the existence of the global attractor for the problem \eqref{eq:CP}. As for the scalar problem, we additionally assume that $\phi:[0,1)\to [0,\infty)$ satisfies the inequality in \eqref{eq:adphi}.

\begin{theorem}\label{GAC} 
	Assume that  $\phi:[0,1)\to [0,\infty)$ satisfies  \eqref{eq:P1}--\eqref{eq:P3} and \eqref{eq:adphi}, and $f, g$ satisfy  \eqref{eq:FG1}-\eqref{eq:FG2}. 
	Then, the semigroup $\{\mathcal{T}(t)\}_{t\geq 0}$ generated by problem \eqref{eq:CP} has a global attractor $\Tilde{\mathcal{A}}$ in $\mathcal{Y}$ endowed with the $L^{1}(\Omega)\times L^{1}(\Omega)$-topology, which has the form $\Tilde{\mathcal{A}}=\omega(\mathcal{Y})$.
\end{theorem}
\begin{remark}
	Theorem~\ref{GAC} generalizes the existence result in \cite{Efendiev2009} for the global attractor for the biofilm model \eqref{eq:BM}. In fact, we consider the significantly broader class of systems \eqref{eq:CP}.
\end{remark}


Next, we recall the definition of an exponential attractor. 
\begin{definition}
	A set $\mathcal{M}\subset E$ is  an exponential attractor for the semigroup $\{S(t)\}_{t\geq 0}$ in a Banach space $E$ iff
	\begin{itemize}
		\item[ \textbf{i)}] $\mathcal{M}$ is a compact subset of $E$,
		\item[\textbf{ii)}] $\mathcal{M}$ is positively invariant, that is, $S(t)\mathcal{M}\subset \mathcal{M},\ \forall\ t\geq 0 $,
		\item[\textbf{iii)}] $\mathcal{M}$ exponentially attracts all bounded subsets, that is, for any bounded subset $D\subset E$ \begin{equation*}
			d_{E}(S(t)D,\mathcal{M})\leq Q(\lVert D\rVert_{E})e^{-\alpha t},
		\end{equation*}
		where the positive constant $\alpha$ and the monotonic function $Q$ are independent of $D$,\item[\textbf{iv)}] $\mathcal{M}$ has finite fractal dimension in $E$, i.e.,
		\begin{equation*}
			\text{dim}_{f}^E(\mathcal{M})\leq C<\infty,
		\end{equation*}for some  constant $C>0$. 
	\end{itemize}
\end{definition}
The fractal dimension of a (pre)compact set $K$ in a metric space $X$ is defined as
\begin{equation*}        \text{dim}_{f}^X(K):=\limsup\limits_{\varepsilon\to 0}\frac{\ln(N_{\epsilon}^X(K))}{\ln(1/\varepsilon)},
\end{equation*}
where $N_{\epsilon}^X(K)$ is the minimal number of balls in $X$ with radius $\varepsilon$ and center in $K$ needed to cover $K$.    

To prove the existence of an exponential attractor for the scalar problem \eqref{eq:SP}, we additionally suppose that the function $\phi:I\to\mathbb{R}$ and the partial derivative $f_{u}$ of the function $f:\Omega\times I\to \mathbb{R}$ satisfy
\begin{equation}
	\phi\in C^{2}(I),\qquad  f_{u}\ \text{is continuous with} \ f_{u}(\cdot,0)<0 \ \text{in}\ \Omega.\label{eq:signcondf}
\end{equation}
We assume that the solutions are globally Hölder continuous with exponent $\alpha>0$, i.e. for any solution $u$ of the problem \eqref{eq:SP} and compact interval $\widetilde I\subset(0,\infty)$ there exists a positive constant $M$ such that 
\begin{align}
	\lVert u\rVert_{C^{\alpha}(\widetilde I\times \overline{\Omega})}\leq M. \label{eq:HC}
\end{align}

\begin{remark}\label{HC}
	The interior Hölder continuity of local weak solutions of \eqref{eq:SP} that are bounded away from $1$ was shown in \cite{Muller} under the assumption that
	\begin{equation*}
		C_{1}|z|^{p-1}\leq \phi'(z)\leq C_{2}|z|^{p-1}\qquad \forall z \in[0,\varepsilon],
	\end{equation*}
	for some constants $C_{1},C_{2}>0$, $\varepsilon\in(0,1)$ and $p>1$ 
	(see \cite[Theorem 1.2]{Muller}). Note that weak solutions $u$ are bounded away from $1$ if the initial data $u_{0}$ has this property. It is expected that these solutions satisfy the global Hölder continuity assumption \eqref{eq:HC}, i.e. they are H\"older continuous up to the boundary. In fact, as pointed out by the author in \cite[Section 1.7]{MullerThesis}, the proof should extend by using the methods in \cite[Chapter III, \S11 and \S12]{DiBenedetto} if the boundary of the domain, the boundary data and initial data are Hölder continuous.    
\end{remark}

Now, we state the main result on the existence of an exponential attractor for the scalar problem.
\begin{theorem}\label{EAS}
	Under the assumptions of Theorem~\ref{GAS} and \eqref{eq:signcondf}-\eqref{eq:HC}, there exists an exponential attractor $\mathcal{M}$ in $\mathcal{X}$ endowed with the topology of $L^{1}(\Omega)$, whose fractal dimension is bounded by\begin{equation}
		\text{dim}_{f}^{L^{1}(\Omega)}(\mathcal{M})\leq \frac{1}{\alpha} +\log_{1/\eta_{1}}\left(N_{\eta_{2}}^W\left(B_{1}^{V}(0)\right)\right), \label{eq:upperboundexp1}
	\end{equation}
	where $W=C(J\times \overline{\Omega}_1)$ and $V=C^\alpha(J\times \overline{\Omega}_2)$ for open subsets $\Omega_{1}$ and $\Omega_{2}$ such that $\Omega_{1}\subset\Omega_{2}\subset\Omega$ and a compact time interval $J\subset\mathbb{R}$. The constants $\eta_{1},\eta_{2}\in(0,1)$ are determined by smoothing and decay estimates of the semigroup and $\alpha>0$ by \eqref{eq:HC}.   
	
	As a consequence, the global attractor $\mathcal{A}\subset \mathcal{M}$ and has finite fractal dimension.
\end{theorem}

Here, the set $B_{r}^{E}(x_{0})$ denotes a ball in a metric space $E$ of radius $r>0$ with center $x_{0}\in E$.

\begin{remark}
	The (Kolmogorov) $\varepsilon$-entropy of a precompact subset $M$ of a Banach space $E$
	is defined as 
	$$
	\mathcal H_\varepsilon^E(M)=\log_2(N_\varepsilon^E(M)).
	$$
	It was introduced by Kolmogorov and Tihomirov in \cite{Kolmogorov}.
	The order of growth of $\mathcal H_\varepsilon^E(M)$ as $\varepsilon$ tends to zero is a measure for the massiveness of the set $M$ in $E$.
	For certain function spaces, precise estimates are known. For instance, if $\Omega=(0,1)^n$ and $\alpha\in(0,1)$ then 
	$$
	c_1\varepsilon^
	{-\frac{n}{\alpha}} \leq\mathcal  H_\varepsilon^{C(\overline{\Omega})}\left(B^{C^\alpha(\overline{\Omega})}_{1}(0)\right)\leq c_2\varepsilon^
	{-\frac{n}{\alpha}},
	$$
	for some constants $c_1,c_2>0,$ see \cite{Kolmogorov}. Such results have been generalized, e.g. for compact, connected subsets of $\mathbb{R}^n$. However, we cannot apply these estimates in Theorem \ref{EAS} as the subsets $\Omega_1$ and $\Omega_2$ are not necessarily connected.    
\end{remark}

\begin{remark}
	To prove the existence of an exponential attractor for continuous time semigroups $\{S(t)\}_{t\geq 0}$ one first constructs an exponential attractor $\mathcal{M}^d$ for the discrete time semigroup $\{S(nT)\}_{n\in\mathbb{N}_0}$, for some $T>0$. If the semigroup is H\"older continuous in time with exponent $\alpha\in(0,1]$ an exponential attractor for the continuous time semigroup is then typically obtained by 
	$$
	\mathcal{M}=\bigcup_{t\in[T,T+1]}S(t)\mathcal{M}^d,
	$$
	and its fractal dimension is bounded by 
	$$
	\text{dim}_{f}(\mathcal{M})\leq \frac{1}{\alpha}\left(1+\text{dim}_{f}(\mathcal{M}^d)\right),
	$$
	see \cite{CzajaEfendiev, Eden, EfendievZelik2008}. In the proof of Theorem \ref{EAS}, we use the slightly different construction of the continuous time exponential attractor in \cite{SonnerPhD}, see also \cite{CzajaSonner}, and obtain the improved bound
	$$
	\text{dim}_{f}(\mathcal{M})\leq \frac{1}{\alpha}+\text{dim}_{f}(\mathcal{M}^d).
	$$
\end{remark}

Finally, we extend the existence result for exponential attractors to the coupled system \eqref{eq:CP} under the assumptions that the function $\phi:[0,1)\to[0,\infty)$ and the partial derivative $f_{u}$ of the function  $f:\Omega\times [0,1)\times [0,1]\to \mathbb{R}$ satisfy 
\begin{align} 
	\phi\in C^{2}([0,1)),\quad \ f_{u}\ \text{is continuous with}\ f_{u}(\cdot,0,v)<0 \ \text{in}\ \Omega,\ \forall v\in [0,1].\label{eq:signcondf2}      
\end{align}
We assume that the solutions are globally Hölder continuous with exponent $\Tilde{\alpha}>0$, i.e. for any solution $(u,v)$ of $\eqref{eq:CP}$ and compact interval $\widetilde I\subset(0,\infty)$ there exists a positive constant $\Tilde{M}$ such that 
\begin{align}
	||(u,v)||_{C^{\Tilde{\alpha}}(\widetilde I\times \overline{\Omega})\times C^{\Tilde{\alpha}}(\widetilde I\times \overline{\Omega})}\leq \Tilde{M}.\label{eq:HC2}
\end{align}

\begin{remark}
	The interior Hölder regularity for the scalar problem \eqref{eq:SP} shown in \cite{Muller} extends to the coupled problem \eqref{eq:CP}. Indeed, the first equation is of the form \eqref{eq:SP} and hence, solutions are Hölder continuous. On the other hand, the second equation of \eqref{eq:CP} is a non-degenerate semilinear parabolic equation for which classical results on interior Hölder continuity are applicable (see e.g., \cite{Ladyzhenskaya}). Therefore, as explained in Remark~\ref{HC}, it is expected that the global Hölder continuity also holds for the coupled problem \eqref{eq:CP}.
\end{remark}

\begin{theorem}\label{EAC}
	Under the assumptions of Theorem~\ref{GAC} and \eqref{eq:signcondf2}-\eqref{eq:HC2}, there exists an exponential attractor $\mathcal{M}$ in $\mathcal{Y}$ endowed with the topology of $L^{1}(\Omega)\times L^{1}(\Omega)$, whose fractal dimension is bounded by\begin{align}
		\text{dim}_{f}^{L^{1}(\Omega)\times L^{1}(\Omega)}(\mathcal{M})\leq \frac{1}{\Tilde{\alpha}}+\log_{1/\chi_{1}}\left(N^W_{\chi_{2}}\left(B_{1}^{V}(0,0)\right)\right),\label{eq:uppervoundexp2}
	\end{align} 
	where $W=C(J\times \overline{\Omega}_{1})\times C(J\times \overline{\Omega})$ and  
	$V=C^{\Tilde{\alpha}}(J\times \overline{\Omega}_{2})\times C^{\Tilde{\alpha}}(J\times \overline{\Omega})$ for open subsets $\Omega_{1}, \Omega_{2}$ such that $\Omega_{1}\subset\Omega_{2}\subset\Omega$ and a compact time interval $J\subset\mathbb{R}$. 
	The constants $\chi_{1},\chi_{2}\in(0,1)$ are determined by smoothing and decay estimates of the semigroup and $\alpha>0$ by \eqref{eq:HC2}.
	
	As a consequence, the global attractor $\Tilde{\mathcal{A}}\subset \mathcal{M}$ and  has finite fractal dimension.
\end{theorem}

\section{Existence of Global Attractors}\label{sec:GA}
In this section, we prove the existence of global attractors for the semigroups generated by problems \eqref{eq:SP} and \eqref{eq:CP}, respectively, i.e. we prove Theorems~\ref{GAS} and \ref{GAC}. First, we consider the scalar problem \eqref{eq:SP} and show the existence of the global attractor for the semigroup $\{\mathcal{S}(t)\}_{t\geq 0}$ in the phase space $\mathcal{X}\subset L^{1}(\Omega)$ defined in Section 2. 
This requires uniform bounds for $\phi(u)$, where $u$ is the weak solution in Theorem~\ref{WPS}. To rigorously justify these estimates, we  consider smooth non-degenerate approximations and derive uniform estimates for the approximate solutions exploiting arguments in \cite{MullerSonner, MullerThesis,Efendiev2009}.
Then, by passing to the limit in the approximation parameter, we obtain suitable bounds and smoothing estimates for the weak solution of the original problem and can define the semigroup in $\mathcal{X}$. Since $\mathcal{X}$ is 
positively invariant and bounded, it is certainly a bounded absorbing set for the semigroup. Hence, it suffices to show the asymptotic compactness of the semigroup which we prove using the uniform estimates for the approximate solutions and by adapting arguments in \cite[\S 3.5]{Feireisl} to our case. 
The results are then extended for the coupled problem.

We first assume that the initial data $u_0\in L^1(\Omega)$ satisfies
\begin{align}\label{eq:ass_smooth_u0}
    \|u_0\|_{L^\infty(\Omega)}\leq 1-\eta,\quad \phi(u_0)\in V,
\end{align}
for some $\eta\in(0,1)$, and then generalize the results for initial data $u_0\in\mathcal{X}$.
We consider the following non-degenerate auxiliary problem
\begin{equation}
		\begin{cases}
\partial_{t}u_{R}=\Delta\phi_{R}(u_{R})+f_R(\cdot,u_{R}),& \text{in}\ (0,T]\times\Omega,\\        
u_{R}=u_{0,R} & \text{in}\ \{0\}\times\Omega,\\
        \phi(u_{R})=0& \text{on}\ (0,T]\times\partial\Omega,
        \end{cases} \label{eq:app}
\end{equation}
where $\phi_R,f_R$ and $u_{0,R}$ are smooth approximations for $\phi, f$ and $u_0$ with the following properties:
\begin{itemize}
    \item[$(A_1)$] 
    The functions $\phi_R\in C^\infty(\mathbb{R})$ are convex on $[0,1)$ and concave on $(-1,0]$, $\phi_R\to\phi$ uniformly on compact subsets of $I$, there exist constants $m_R,M_R>0$ such that 
    $$m_R\leq \phi_R'\leq M_R,
    $$
    and $\phi_R$ is Lipschitz continuous on compact subsets of $I$, where the Lipschitz constant is independent of $R$. Moreover, $|\phi_{R}(z)|\leq |\phi(z)|+\frac{1}{R}|z|$ for any $z\in [-1+1/R,1-1/R]$.
    \item[$(A_2)$] 
    The functions $f_R\in C^\infty(\overline{\Omega}\times I)$ are Lipschitz continuous, 
    $$\sup_{z\in I}\|f_R(\cdot, z)-f(\cdot, z)\|_{L^2(\Omega)}\to 0$$
    and $\|f_R\|_{L^\infty(\Omega\times I)}\leq \|f\|_{L^\infty(\Omega\times I)}.$
    \item[$(A_3)$]
    The functions $u_{0,R}\in C^\infty_c(\Omega)$ are such that $u_{0,R}\to u_0$ in $H^1(\Omega)$ and $\|u_{0,R}\|_{L^\infty(\Omega)}\leq \|u_{0}\|_{L^\infty(\Omega)}\leq 1-\eta.$
\end{itemize}
The smooth approximations for $\phi$ can be obtained by suitably cutting of the singularity and degeneracy and then taking the convolution with a standard mollifier. Similarly, the smooth approximations for $f$ and $u_0$ are constructed.

From now on we denote by $C>0$ a generic constant in the estimates that may vary in each occurrence and from line to line. If a constant depends on a parameter $\varepsilon$ we indicate it by writing $C_\varepsilon$. Again, this constant $C_\varepsilon$  may vary in each occurrence and from line to line.

\begin{proposition}\label{pro:ext_approx}
Let $\Omega$ be a smooth domain, the initial data satisfy \eqref{eq:ass_smooth_u0} and $(A_1)$--$(A_3)$ hold. 
Then, for any $T>0$ there exists a unique classical solution $u_{R}\in C^\infty((0,T]\times\overline{\Omega})$ of 
\eqref{eq:app}.
\\
Moreover, there exists $\delta\in(0,1)$, depending on $\eta$, such that for sufficiently large $R_\eta>0,$ we have
    $$
    -1+\delta\leq u_R\leq 1-\delta\ \text{ in } \Omega_T\qquad \forall R\geq R_\eta.
    $$
\end{proposition}
\begin{proof}
 The existence and uniqueness of classical solutions of the approximate problems with the stated regularity follows from the theory for non-degenerate quasilinear parabolic equations with smooth coefficients, see, e.g., Theorems 14.4, 14.6 and Corollary 14.7 in \cite{Amann}. 
\\
To derive the bounds for the approximate solutions, we use comparison principles. To this end, we consider the auxiliary problem
\begin{align}\label{eq:ell_aux}
 	\begin{split}
-\Delta w&=c_1\qquad \text{in}\ \Omega,\\
 w&=c_2 \qquad \text{on}\ \partial\Omega,
    \end{split} 
\end{align}
where $c_1=\|f\|_{L^\infty(\Omega\times I)}$ and $c_2=\phi(1-\frac{\eta}{2})$. 
There exists a unique solution $w\in C^\infty(\overline{\Omega})$ of this boundary value problem and the maximum principle implies that $w\geq c_2.$ 
\\
We now set 
$$
\delta=1-\phi^{-1}(\|w\|_{L^\infty(\Omega)}+1)
$$
and show that there exists $R_\delta>0$ such that 
\begin{align}\label{eq:aux_bounds}
    1-\eta\leq \phi_R^{-1}(w)\leq 1-\delta\qquad \forall R\geq R_\delta.
\end{align}
Indeed, $\phi_R$ converges uniformly to $\phi$ on the compact interval $I_\delta=[0,1-\delta]=[0,\phi^{-1}(\|w\|_{L^\infty(\Omega)}+1)]$. Hence, for $R_0>0$ large enough we have 
$\phi_R^{-1}([0,\|w\|_{L^\infty(\Omega)}])\subset I_\delta\subset [0,1)$ for all $R\geq R_0$.
Consequently, $\phi_R^{-1}(w)\leq 1-\delta$ for all $R\geq R_0$. To derive the lower bound in \eqref{eq:aux_bounds}, we recall that  $w\geq c_2=\phi(1-\frac{\eta}{2}).$ Let $R_\eta\geq R_0$ be large enough such that $\|\phi-\phi_R\|_{C(I_\delta)}\leq \phi(\frac{\eta}{2})$ for all $R\geq R_\eta.$ The 
convexity of $\phi$ on $[0,1)$ implies the superadditivity of $\phi$ on $[0,1)$. Therefore, we conclude that  
$$
c_2=\phi\left(1-\frac{\eta}{2}\right)\geq \phi(1-\eta)+\phi\left(\frac{\eta}{2}\right)\geq \phi_R(1-\eta),
$$
and consequently,  $\phi_R^{-1}(w)\geq \phi_R^{-1}(c_2)\geq 1-\eta$ for all $R\geq R_\eta$, which proves \eqref{eq:aux_bounds}.
\\
Finally, we observe that for $\bar u_R=\phi_R^{-1}(w)$ we have
\begin{align*}
    -(\bar u_R)_t+\Delta\phi_R(\bar u_R)&=\Delta w =-c_1=-\|f\|_{L^\infty(\Omega\times I)}\\
    &\leq -f_R(\cdot, u_R)=-(u_R)_t+\Delta\phi_R(u_R),
\end{align*}
and $\bar u_R\geq 1-\eta\geq \|u_{0,R}\|_{L^\infty(\Omega)}\geq 0$ on $((0,T]\times\partial\Omega)\cup (\{0\}\times\Omega).$
Hence, the parabolic comparison principle implies that $u_R\leq \bar u_R\leq 1-\delta.$
\\
This proves the upper bound in Proposition \ref{pro:ext_approx}. The lower bound is obtained analogously, using $c_1=-\|f\|_{L^\infty(\Omega)}$ and 
$c_2=\phi(-1+\frac{\eta}{2})$ in the auxiliary problem \eqref{eq:ell_aux} and the concavity of $\phi$ on $(-1,0]$. 
\end{proof}

The following lemma provides uniform estimates for $\phi_R$.
\begin{lemma}\label{pro:H1bounds}
    Let the assumptions in Proposition \ref{pro:ext_approx} hold and $u_R$ be the approximate solutions. Then, for any $T>0$, the following estimate holds
\begin{align*}
\lVert \phi_{R}(u_{R}(T))\rVert_{H^{1}(\Omega)}^{2}+
\int_{T}^{T+1}\langle|\partial_{t}u_{R}(t)|^{2},\phi'_{R}(u_{R}(t))\rangle dt\leq C(\lVert\phi_{R}(u_{0,R})\rVert_{H^{1}(\Omega)}^{2}+1),
\end{align*}
where  the constant $C>0$ is independent of $R$  and $\eta$.
\end{lemma}
\begin{proof}We multiply $\eqref{eq:app}_{1}$ by $\phi_{R}(u_{R}(t))$ and integrate over $x\in \Omega$. Then, using integration by parts and the boundedness of $f_{R}(\cdot,u_{R})$, together with Poincare's and Young's inequalities, we obtain
     \begin{align*}
        \partial_{t}\langle\Phi_{R}(u_{R}(t)),1\rangle+\lVert \nabla\phi_{R}(u_{R}(t))\rVert_{L^{2}(\Omega)}^{2}&=\langle f_{R}(\cdot,u_{R}(t)),\phi_{R}(u_{R}(t))\rangle\\&\leq
        C_{\varepsilon}+\varepsilon C\lVert \nabla\phi_{R}(u_{R}(t))\rVert_{L^{2}(\Omega)}^{2} ,
    \end{align*}
    for any $\varepsilon>0$, where 
    \begin{align*}
        \Phi_{R}(z):=\int_{0}^{z}\phi_{R}(s)ds.
    \end{align*}
    Hence,  choosing $\varepsilon>0$ small enough, it follows that 
    \begin{align*}
          \partial_{t}\langle\Phi_{R}(u_{R}(t)),1\rangle+\lambda\lVert \nabla\phi_{R}(u_{R}(t))\rVert_{L^{2}(\Omega)}^{2}\leq C, 
    \end{align*}
    where $C$ and $\lambda$ are positive constants independent of $R$ and $u_{0,R}$.

    Next, we multiply $\eqref{eq:app}_{1}$ by $\partial_{t}\phi_{R}(u_{R}(t))$ and integrate over $x\in\Omega$. Then, integration by parts implies that
    \begin{align*}
        &\langle|\partial_{t}u_{R}(t)|^{2},\phi'_{R}(u_{R}(t))\rangle+\partial_{t}\left[\frac{1}{2}\lVert \nabla \phi_{R}(u_{R}(t))\rVert_{L^{2}(\Omega)}^{2}-\langle\int_{0}^{u_{R}(t)}f_{R}(\cdot,v)\phi_{R}'(v)dv,1\rangle\right]=0.
    \end{align*}
    Adding these two estimates, we obtain 
    \begin{align}
       \partial_{t}Q(t)+\lambda\lVert \nabla\phi_{R}(u_{R}(t))\rVert_{L^{2}(\Omega)}^{2}+\langle|\partial_{t}u_{R}(t)|^{2},\phi'_{R}(u_{R}(t))\rangle\leq C,\label{eq:ineq-3}
    \end{align}
    where
    \begin{align*}
        Q(t)&=\frac{1}{2}\lVert \nabla \phi_{R}(u_{R}(t))\rVert_{L^{2}(\Omega)}^{2}-\langle\int_{0}^{u_{R}(t)}f_{R}(\cdot,v)\phi'_{R}(v)dv,1\rangle+\langle\Phi_{R}(u_{R}(t)),1\rangle\\
        &=\frac{1}{2}\lVert \nabla \phi_{R}(u_{R}(t))\rVert_{L^{2}(\Omega)}^{2}+q(t),
    \end{align*}
    and $C$ is independent of $R$. Using the boundedness of $u_R, f_R(\cdot,u_{R})$ and Hölder's and Young's inequalities, we observe that
    \begin{align*}
        |q(t)|&\leq C_\varepsilon+\varepsilon\lVert\phi_{R}(u_{R}(t))\rVert^{2}_{L^{2}(\Omega)}.
    \end{align*}
    Hence, choosing $\varepsilon$ small enough and using Poincare's inequality, we conclude that 
    \begin{align}
        C(\lVert \nabla\phi_{R}(u_{R}(t))\rVert_{L^{2}(\Omega)}^{2}-1)\leq Q(t)\leq C(\lVert \nabla\phi_{R}(u_{R}(t))\rVert_{L^{2}(\Omega)}^{2}+1),\label{eq:ineq-4}
    \end{align}where the constant $C$ is independent of $R$. 
   Together with \eqref{eq:ineq-3}, we therefore obtain 
   \begin{align}
        & \partial_{t}Q(t)+\lambda Q(t)+\langle|\partial_{t}u_{R}(t)|^{2},\phi'_{R}(u_{R}(t))\rangle\leq C,\label{eq:ineq-5}
    \end{align} 
    and since $\langle|\partial_{t}u_{R}(t)|^{2},\phi'_{R}(u_{R}(t))\rangle$ is nonnegative, we have \begin{align*}
        & \partial_{t}Q(t)+\lambda Q(t)\leq C.
    \end{align*}
    Hence, Gronwall's inequality implies that
    \begin{align}
        Q(T)&\leq Q(0)e^{-\lambda T}+C\int_{0}^{T}e^{-\lambda(T-s)}ds\leq Q(0)+C. \label{eq:ineq-6}
    \end{align}
    On the other hand, integrating \eqref{eq:ineq-5} over $t\in [T,T+1]$, we obtain
    \begin{align*}
        Q(T+1)-Q(T)+\lambda\int_{T}^{T+1}Q(s)ds+\int_{T}^{T+1}\langle|\partial_{t}u_{R}(s)|^{2},\phi'_{R}(u_{R}(s))\rangle ds\leq C,
    \end{align*}
    which, by \eqref{eq:ineq-6} and \eqref{eq:ineq-4}, yields
    \begin{align}
\int_{T}^{T+1}\langle|\partial_{t}u_{R}(s)|^{2},\phi'_{R}(u_{R}(s))\rangle ds\notag&\leq Q(T)-Q(T+1)-\lambda\int_{T}^{T+1}Q(s)ds+C\\&\leq Q(0)+C .\label{eq:ineq-7}
    \end{align}Summing \eqref{eq:ineq-6} and \eqref{eq:ineq-7}, we get \begin{align*}
       Q(T)+
\int_{T}^{T+1}\langle|\partial_{t}u_{R}(t)|^{2},\phi'_{R}(u_{R}(t))\rangle dt\leq Q(0)+C,
    \end{align*}which, together with \eqref{eq:ineq-4} and Poincare's inequality, gives the required inequality.
\end{proof}

The next lemma provides a uniform smoothing estimate.
\begin{lemma}\label{pro:H1est}
    Let the assumptions in Proposition \ref{pro:ext_approx} hold and $u_R$ be the approximate solutions. Then there exist positive constants $\kappa$ and $C$, which are independent of $R$ and  $\eta$, such that
    \begin{align}
        \lVert \phi_{R}(u_{R}(t))\rVert_{H^{1}(\Omega)}^{2}\leq C\frac{t^{\kappa}+1}{t^{\kappa}} \quad \text{for all}\ t>0.\label{eq:fphi}
    \end{align}
\end{lemma}
\begin{proof}First, we remark that it suffices to prove  \eqref{eq:fphi} for $t\leq 1$. 
Indeed, if we have 
 \begin{align}
        \lVert \phi_{R}(u_{R}(t))\rVert_{H^{1}(\Omega)}^{2}\leq C\frac{t^{\kappa}+1}{t^{\kappa}} \quad \text{for}\  0<t\leq 1,\label{eq:fphi'}
    \end{align}
    then for $t>1$, by Lemma \ref{pro:H1bounds} and Poincare's inequality, we have
    \begin{align*}
        \lVert \phi_{R}(u_{R}(t))\rVert_{H^{1}(\Omega)}^{2}
     \leq C(\lVert \phi_{R}(u_{R}(1))\rVert_{H^{1}(\Omega)}^{2}+1)
    \leq C,
    \end{align*}
    which, together with \eqref{eq:fphi'}, gives \eqref{eq:fphi} for any $t>0$.
    
Let now $t\leq 1$. We multiply $\eqref{eq:app}_{1}$ by $|\phi_{R}(u_{R}(t))|^{\delta-1}\phi_{R}(u_{R}(t))$ for $\delta>0$ to be specified later and integrate over $\Omega$. Then, using integration by parts, we have
\begin{align}
\begin{split}
        \partial_{t}\langle\Phi_{R,\delta}(u_{R}(t)),1\rangle&+C\int_{\Omega} |\nabla(\phi_{R}(u_{R}(t)))^{\frac{\delta+1}{2}}|^{2}\\
        &=\int_{\Omega} f_R(\cdot,u_{R}(t))|\phi_{R}(u_{R}(t))|^{\delta-1}\phi_{R}(u_{R}(t)), \label{eq:0}
\end{split}
\end{align}
    where 
    \begin{align*}
        \Phi_{R,\delta}(u):=\int_{0}^{u} |\phi_{R}(v)|^{\delta-1}\phi_{R}(v)dv. 
    \end{align*}
  To estimate the term on the right hand side, we use the boundedness of $f_R(\cdot,u_{R})$ and apply Hölder's, Young's and  Poincare's inequalities,  
   \begin{align*}
     &\int_{\Omega}f_R(\cdot,u_{R}(t))|\phi_{R}(u_{R}(t))|^{\delta-1}\phi_{R}(u_{R}(t)) \leq C \int_{\Omega} |\phi_{R}(u_{R}(t))|^{\delta}\\
     &\leq C\left[\int_{\Omega} |\phi_{R}(u_{R}(t))|^{\delta+1} dx\right]^{\frac{\delta}{\delta+1}} \notag
     \leq \varepsilon \int_{\Omega} |\phi_{R}(u_{R}(t))|^{\delta+1}+C_\varepsilon\\
     &=\varepsilon \int_{\Omega} ||\phi_{R}(u_{R}(t))|^{\frac{\delta+1}{2}}|^{2}+C_\varepsilon
     \leq \varepsilon C\int_{\Omega} |\nabla |\phi_{R}(u_{R}(t)))|^{\frac{\delta+1}{2}}|^{2}+C_\varepsilon,
    \end{align*} 
    for any $\varepsilon>0$. Choosing $\varepsilon$ sufficiently small, this estimate,  together with \eqref{eq:0} and Poincare's inequality, yields 
    \begin{align*}
        \partial_{t}\langle\Phi_{R,\delta}(u_{R}(t)),1\rangle
+\gamma\lVert |\phi_{R}(u_{R}(t))|^{\frac{\delta+1}{2}}\rVert_{H^{1}(\Omega)}^{2}\leq C 
    \end{align*}
    for some constant $\gamma>0$. Then, since $
        \lVert |\phi_{R}(u_{R})|^{\frac{\delta+1}{2}}\rVert_{L^{2}(\Omega)}^{2}=\lVert |\phi_{R}(u_{R})|^{\delta+1}\rVert_{L^{1}(\Omega)}
$, we obtain  
    \begin{align}   \partial_{t}\langle\Phi_{R,\delta}(u_{R}(t)),1\rangle
+\gamma\lVert \nabla|\phi_{R}(u_{R}(t))|^{\frac{\delta+1}{2}}\rVert_{L^{2}(\Omega)}^{2}+\gamma \lVert |\phi_{R}(u_{R}(t))|^{\delta+1}\rVert_{L^{1}(\Omega)}\leq C.\label{eq:3}
      \end{align}   
       
Next, let us integrate \eqref{eq:3} with respect to $t$. Then, considering the inequality 
\begin{align*}
\Phi_{R,\delta_{0}}(u_{0,R})&\leq\int_{0}^{|u_{0,R}|} |\phi_{R}(v)|^{\frac{1}{a}}dv
\leq \int_{0}^{|u_{0,R}|} [\phi(v)+\frac{1}{R}v]^{\frac{1}{a}}dv\\
&\leq C\int_{0}^{|u_{0,R}|} [(1-v)^{1-a}+1]^{\frac{1}{a}}dv\leq C\int_{0}^{|u_{0,R}|} [(1-v)^{\frac{1}{a}-1}+1]dv\notag\\
&\leq C[a(1-(1-|u_{0,R}|)^{\frac{1}{a}})+|u_{0,R}|]\leq C,
\end{align*}
due to \eqref{eq:adphi} and \eqref{eq:ass_smooth_u0}, where $\delta_{0}:=a^{-1}$, we obtain 
\begin{align}
    \int_{0}^{1} (\lVert \nabla|\phi_{R}(u_{R}(t))|^{\frac{\delta_{0}+1}{2}}\rVert_{L^{2}(\Omega)}^{2}+\lVert |\phi_{R}(u_{R}(t))|^{\delta_{0}+1}\rVert_{L^{1}(\Omega)})dt\leq C. \label{eq:5}
\end{align}

Now, let us assume $\delta\geq \delta_{0}$. We multiply \eqref{eq:3} by $t^{N}$ for $N\in \mathbb{N}$ and integrate over $\Omega$. Then, using integrating by parts and the inequality
\begin{align*}
    \Phi_{R,\delta}(u_{R}(t))\leq \int_{0}^{|u_{R}(t)|}|\phi_{R}(v)|^{\delta}dv\leq |u_{R}(t)||\phi_{R}(u_{R}(t))|^{\delta}\leq |\phi_{R}(u_{R}(t))|^{\delta},
\end{align*}
we get
\begin{align}
\begin{split}
&\int_{0}^{1} t^{N}(\lVert \nabla|\phi_{R}(u_{R}(t))|^{\frac{\delta+1}{2}}\rVert_{L^{2}(\Omega)}^{2}+\lVert |\phi_{R}(u_{R}(t))|^{\delta+1}\rVert_{L^{1}(\Omega)})dt\\
&\leq C\left[1+\int_{0}^{1}Nt^{N-1}\langle\Phi_{R,\delta}(u_{R}(t)),1\rangle dt\right] \\
&\leq C\left[1+\int_{0}^{1}t^{N-1}\lVert \phi_{R}(u_{R}(t)\rVert^{\delta}_{L^{\delta}(\Omega)}dt\right].\label{eq:6}
\end{split}
\end{align}
Let us take $\delta=1$ in \eqref{eq:6}. Then, by Hölder's and Young's inequalities, we have
\begin{align*}
    &\int_{0}^{1} t^{N}(\lVert \nabla(\phi_{R}(u_{R}(t)))\rVert_{L^{2}(\Omega)}^{2}+\lVert \phi_{R}(u_{R}(t))\rVert_{L^{2}(\Omega)}^{2})dt\\
    &\leq C\left[1+\int_{0}^{1}t^{N-1}\lVert \phi_{R}(u_{}(t))\rVert_{L^{1}(\Omega)}dt\right]
    \leq C\left[1+\int_{0}^{1}\lVert \phi_{R}(u_{R}(t))\rVert^{\delta_{0}+1}_{L^{\delta_{0}+1}(\Omega)}dt\right],
\end{align*}
which, together with \eqref{eq:5}, yields \begin{align}
\int_{0}^{1} t^{N}\lVert \phi_{R}(u_{R}(t))\rVert_{H^{1}(\Omega)}^{2}dt
    \leq C. \label{eq:7}
\end{align} 
It follows from \eqref{eq:7} that for any $t\in(0,1]$
\begin{align}
   \int_{t/2}^{t} s^{N}\lVert \phi_{R}(u_{R}(s))\rVert_{H^{1}(\Omega)}^{2}ds\leq \int_{0}^{1} t^{N}\lVert \phi_{R}(u_{R}(t))\rVert_{H^{1}(\Omega)}^{2}dt
    \leq C. \label{eq:9}
\end{align}
On the other hand, considering Lemma \ref{pro:H1bounds}, by the mean value theorem, there exists $T_{0}\in [t/2,t]$ such that
\begin{align*}
    \int_{t/2}^{t} s^{N}\lVert \phi_{R}(u_{R}(s))\rVert_{H^{1}(\Omega)}^{2}ds=\frac{t}{2}T_{0}^{N}\lVert \phi_{R}(u_{R}(T_{0}))\rVert_{H^{1}(\Omega)}^{2}.
\end{align*}
Then, by \eqref{eq:9} and using $T_{0}\geq t/2$, we have
\begin{align*}
        \lVert \phi_{}(u_{R}(T_{0}))\rVert_{H^{1}(\Omega)}^{2}\leq \frac{2}{t}CT_{0}^{-N}\leq Ct^{-N-1}.
\end{align*}
Hence, replacing the initial time $t=0$ by $t=T_{0}$ in Lemma \ref{pro:H1bounds}, we obtain 
\begin{align*}
    \lVert \phi_{R}(u_{R}(t)\rVert_{H^{1}(\Omega)}^{2}\leq C(\lVert \phi_{R}(u_{R}(T_{0})\rVert_{H^{1}(\Omega)}^{2}+1)\leq C(t^{-N-1}+1),
\end{align*}
which yields \eqref{eq:fphi} with $\kappa=N+1$. It is important to note that the  constants $C$ in all estimates are independent of $R$ and $\eta$.

\end{proof}

\begin{theorem}\label{thm:existence}
Let the initial data satisfy 
\eqref{eq:ass_smooth_u0}.
Then, the degenerate problem \eqref{eq:SP} possesses a solution $u$ belonging to the class
\begin{align*}
&u\in L^{\infty}((0,\infty)\times\Omega)\cap C([0,\infty);L^{2}(\Omega)),\ 
u_{t}\in L^{2}((0,\infty);H^{-1}(\Omega))\\
&\phi(u)\in L^{\infty}((0,\infty);H^{1}(\Omega))\cap C([0,\infty);L^{2}(\Omega)).
\end{align*}
Furthermore, the solution satisfies $\|u\|_{L^\infty((0,\infty)\times\Omega)}\leq 1-\delta$, for some $\delta\in(0,1)$ depending on $\eta,\Omega, \phi$ and $\|f\|_{L^\infty(\Omega\times I)}$, and the following estimates hold, 
\begin{align}
\lVert \phi(u(t))\rVert_{H^{1}(\Omega)}^{2}&
\leq C(\lVert\phi(u_0)\rVert_{H^{1}(\Omega)}^{2}+1),\label{eq:pointwise2}\\
        \lVert \phi(u(t))\rVert_{H^{1}(\Omega)}^{2}+\|\partial_t u(t)\|^2_{H^{-1}(\Omega)}&
        \leq C\frac{t^{\kappa}+1}{t^{\kappa}}, \label{eq:fphi2}
\end{align}
for all $t>0$, where the constants $C$ and $\kappa$ are independent of $\eta$.    
\end{theorem}
\begin{proof} 
Without loss of generality we assume that $\Omega$ is smooth. The proof extends to a general bounded domain by approximating $\Omega$ with smooth domains $\Omega_k$
such that $\Omega=\bigcup_k\Omega_k$ and applying the subsequent arguments on $\Omega_k.$ The constants in the estimates will be independent of $k$.
\\
Let $u_R$ be the approximate solutions. By Proposition \ref{pro:ext_approx}, we have the uniform bound 
\begin{align}\label{eq:uniformbounds}
|\phi_{R}(u_{R})|+|\phi_{R}'(u_{R})|\leq C \qquad \text{in}\ (0,\infty)\times\Omega.
\end{align}
Moreover, by Lemma \ref{pro:H1bounds}, 
\begin{align}\label{eq:H1bounds}
    \|\phi_{R}(u_{R})\|_{L^{\infty}((0,\infty);H^{1}(\Omega))}+\|\partial_tu_{R}\|_{L^{\infty}((0,\infty);H^{-1}(\Omega))}\leq C,
\end{align}
for some constant $C\geq 0$ independent of $R$. 
Furthermore, we observe that 
$$
\|\partial_t \phi_R(u_R)\|_{L^2((0,T)\times\Omega)}^2=
\| \phi_R'(u_R)\partial_t u_R\|_{L^2((0,T)\times\Omega)}^2\leq C,
$$
by Lemma \ref{pro:H1bounds} and \eqref{eq:uniformbounds}. 
We observe that for all $t\in[0,T]$ the set $\{\phi_R(u(t))\}$ is relatively compact in $L^2(\Omega)$ and the family $\{\phi_R(u_R)\}$ is equicontinuous. Hence, the Arzela-Ascoli Theorem implies that there exists $w\in C([0,T];L^2(\Omega))$ such that 
$$
\phi_R(u_R)\to w\ \text{ in } \ C([0,T];L^2(\Omega)).
$$
along a subsequence. Refining the subsequence, we can show that
$\phi_R(u_R)\to w$ a.e. in $\Omega_T$ and consequently, $|w|\leq 1-\delta$ a.e. in $\Omega_T.$
\\
By \eqref{eq:H1bounds}, $\nabla \phi_R(u_R)$ is uniformly bounded in $L^2(\Omega_T)$ and hence, $\nabla \phi_R(u_R)\rightharpoonup \psi$ in  
$L^2(\Omega_T)$. Combining it with the above arguments, we conclude that $\psi=\nabla w$ and 
$w\in L^2(0,T;H_0^1(\Omega)).$
\\
Finally, we set $u=\phi^{-1}(w)$ and show that $u_R\to u$. We observe that 
$\phi(u_R)\to w$ a.e. in $\Omega_T$ and since $\phi^{-1}:[\phi(-1+\delta),\phi(1-\delta)]\to[-1+\delta,1-\delta]$ is a continuous bijection, we conclude that   $u_R\to u$
and $|u|\leq 1-\delta$ a.e. in $\Omega_T$. Recall that $\phi(0)=0$, $\phi$ is convex on $[0,1)$ and
concave on $(-1,0]$. Hence, $\phi_+$ and $\phi_-$ are convex, where $\phi_{+}:= \max\{\phi, 0\}$ and $\phi_{-}:=-(-\phi)_{+}$ denote the positive and negative parts of $\phi$, respectively, such that $\phi=\phi_{+}-\phi_{-}$. 
Using Jensen's inequality and the superadditivity of $\phi_+$, we obtain 
\begin{align*}
&   \phi_+\left(\fint_\Omega|u_+(t)-u_{R,+}(t)| \right)\leq 
   \fint_\Omega\phi_+(|u_+(t)-u_{R,+}(t)|) \\
   &\leq \fint_\Omega|\phi_+(u_+(t))-\phi_+(u_{R,+}(t))|\\
   &\leq \fint_\Omega|\phi_+(u_+(t))-\phi_{R,+}(u_{R,+}(t))|+\fint_\Omega|\phi_{R,+}(u_{R,+}(t))-\phi_+(u_{R,+}(t))|\\
   &\leq C\left( \|w_+(t)-\phi_{R,+}(u_{R,+}(t))\|_{L^2(\Omega)}+\|\phi_{R,+}-\phi_{+}\|_{C([1-\delta,1+\delta])}\right)\to 0
\end{align*}
uniformly in $t$ as $R\to\infty$, where $\fint_{\Omega}u=\frac{1}{|\Omega|}\int_{\Omega}u$
 denotes the averaged integral of $u$ over $\Omega$.  Applying $\phi^{-1}_+$ and  using the boundedness of $u_+(t)$ and $u_{R,+}(t)$, we conclude that $u_{R,+}\to u_+$ in 
$C([0,T];L^2(\Omega))$. Similarly, we deduce that $u_{R,-}\to u_-$ in 
$C([0,T];L^2(\Omega))$, for details we refer to the proof of Theorem 2.7 in \cite{MullerSonner}. Hence, we obtain $u_R=u_{R,+}-u_{R,-}\to u_{+}-u_{-}= u$ in 
$C([0,T];L^2(\Omega))$.
\\
For the convergence of the term involving  $f_R$, we observe that 
\begin{align*}
    &\|f(\cdot,u)-f_R(\cdot,u_R)\|_{L^2(\Omega_T)}\\
    &\leq 
    \|f(\cdot,u)-f(\cdot,u_R)\|_{L^2(\Omega_T)}+\|f(\cdot,u_R)-f_R(\cdot,u_R)\|_{L^2(\Omega_T)}\\
    &\leq L\|u-u_R\|_{L^2(\Omega_T)}+\sup_{z\in I}\|f(\cdot,z)-f_R(\cdot,z)\|_{L^2(\Omega)}\to 0.
\end{align*}
Hence, we can now pass to the limit $R\to\infty$ in the weak formulation for the approximate problem \eqref{eq:app} and conclude that the limit $u$ is a weak solution of the degenerate problem \eqref{eq:SP}.
\\ 
The fact that the solution belongs to the stated class and that the estimate \eqref{eq:pointwise2} holds are immediate consequences of \eqref{eq:uniformbounds}, \eqref{eq:H1bounds} and Lemma \ref{pro:H1bounds}. The estimate \eqref{eq:fphi2} follows from \eqref{eq:H1bounds} and the smoothing estimate in Lemma \ref{pro:H1est}.
\end{proof}

We now set 
$$\mathcal{X}_0=\{u\in L^1(\Omega): u\text{ satisfies } \eqref{eq:ass_smooth_u0}\}.
$$
Then, by Theorem \ref{WPS} and Theorem \ref{thm:existence} the degenerate problem \eqref{eq:SP} generates a semigroup $\{\mathcal{S}(t)\}_{t\geq 0}$ in $\mathcal{X}_0$ 
that is Lipschitz continuous in $L^1(\Omega)$. 
Indeed, the assumption $\|u_0\|_{L^\infty(\Omega)}\leq 1-\eta$ implies that $\Phi(u_0)\in L^{\infty}(\Omega)$ and hence, the $L^1$-contraction in Theorem \ref{WPS} holds for the solutions considered in Theorem \ref{thm:existence}.
More specifically,
$\mathcal{S}(t)u_0:=u(t)$, where 
$u$ is the solution of \eqref{eq:SP} with initial data $u_0\in\mathcal{X}_0$. This solution $u$ is obtained as limit 
$u(t)=\lim_{R\to\infty}u_R(t)$ in $L^2(\Omega),$
where $u_R$ are the solutions of the non-degenerate approximate problems  \eqref{eq:app}.
\\
Using the $L^1$-contraction in Theorem \ref{WPS}, we can extend $\{\mathcal{S}(t)\}_{t\geq 0}$ to a semigroup $\{\mathcal{S}(t)\}_{t\geq 0}$ in the phase space $\mathcal{X}=\overline{\mathcal{X}_0}$ which is Lipschitz continuous in $L^1(\Omega)$. More specifically, we take $u_0^k\in \mathcal{X}_0$ such that $u_0^k\to u_0$ in $L^1(\Omega)$ and then set 
$$
\mathcal{S}(t)u_0:=\lim_{k\to\infty}\mathcal{S}(t)u_0^k\qquad \text{in}\ L^1(\Omega).
$$
\begin{theorem}\label{thm:limit}
Let $u_{0}\in \mathcal{X}$ and $u(t)=\mathcal{S}(t)u_0.$
Then,  $u\in L^{\infty}((0,\infty)\times\Omega)\cap C([0,\infty),L^{1}(\Omega))$  and for all $t>0$ the set $\{x\in \Omega:|u(t,x)|=1\}$ has Lebesgue measure zero. Moreover, the estimate \eqref{eq:fphi2} holds for $u$ and $u$ is a solution of \eqref{eq:SP}. 
\\
As a consequence, if $u_{0}\in \mathcal{X}$ and $u_R$ is the corresponding solution
of the auxiliary problem \eqref{eq:app}, 
then, the solution $u(t)= S(t)u_0$ of the degenerate problem is the limit 
$u(t) =  \lim_{R\to\infty}u_R(t)$ in $L^1(\Omega)$. 
\end{theorem}
\begin{proof}
Let $u_0^k\in L^1(\Omega)$ satisfy \eqref{eq:ass_smooth_u0} and $u^k_0\to u_0$ in $L^1(\Omega).$ Then, by Theorem \ref{thm:existence} and Theorem  \ref{WPS} the corresponding solutions satisfy $u_k\in C([0,\infty);L^1(\Omega))$ and hence, also $\lim_{k\to\infty}u_k=u\in C([0,\infty);L^1(\Omega)).$ 
    \\
Due to Theorem \ref{WPS}, we have $\|u_k(t)-u(t)\|_{L^1(\Omega)}\to 0$ for all $t\geq 0$ and hence, $u_k(t)\to u(t)$ a.e. in $\Omega.$ 
Therefore, since $\phi$ is monotone, using Poincare's inequality and Lemma \ref{pro:H1est}, we conclude that for any $\delta\in(0,1)$
    $$
    |\{x\in \Omega:|u^k(t,x)|\geq 1-\delta\}|\leq \frac{C_t}{\min\{\phi(1-\delta)^2,\phi(-1+\delta)^2\}},
    $$
    where the constant $C_t$ depends on time, but is independent of $k$. 
    Let $\rho_\delta:\mathbb{R}\to[0,1]$ be a continuous function such that 
    $$
    \rho_\delta(z)=\begin{cases}
        0&z\leq 1-2\delta,\\
        1&z\geq 1-\delta.
    \end{cases}
    $$
    Then, the dominated convergence theorem implies that
    \begin{align*}
    |\{x\in \Omega:u(t,x)\geq 1-\delta\}|&\leq\int_\Omega\rho_\delta(u(t,x))dx    =\lim_{k\to\infty}\int_\Omega\rho_\delta(u^k(t,x))dx\\
    &\leq \limsup_{k\to\infty}|\{x\in \Omega:u^k(t,x)\geq 1-2\delta\}|\leq \frac{C_t}{\phi(1-2\delta)^2}.
    \end{align*}
  Taking the limit $\delta\to 0$ in the above estimate, we conclude that
$|\{x\in \Omega:u(t,x)=1\}|=0.$
Analogously, we can show that $|\{x\in \Omega:u(t,x)=-1\}|=0.$
\\
As in the proof of Theorem 3.5 in \cite{Efendiev2009} it can now be shown that $u$ satisfies estimate \eqref{eq:fphi2} and $u$ is a weak solution of \eqref{eq:SP}.
The last statement is a consequence of the $L^1$-contraction, see the proof of Corollary 3 in \cite{Efendiev2009}.
\end{proof}

We are now in a position to prove the asymptotic compactness of the semigroup.

\begin{lemma}\label{ACS}
 For any $T>0$, the set $\bigcup\limits_{t\geq T} \mathcal{S}(t)\mathcal{X}$ is relatively compact in $L^1(\Omega)$.
\end{lemma}

\begin{proof}
	Let $T>0$ and $u$ be any solution of problem \eqref{eq:SP} with initial data $u_{0}\in \mathcal{X}$. By the estimate \eqref{eq:fphi2} from Theorem~\ref{thm:limit}, we have  \begin{align*}
		\lVert\phi(u(t))\rVert_{H^{1}(\Omega)}\leq \frac{C}{t^{\kappa}}\qquad \forall t>0,
	\end{align*} which implies that there exists $\xi\in (0,T]$ such that
	\begin{align}
		\lVert\phi(u(\xi))\rVert_{H^{1}(\Omega)}\leq \frac{C}{\xi^{\kappa}},\label{eq:Es1}
	\end{align}where $C>0$ is a constant independent of $u_{0}$.
	Now, we define the set
	\begin{align*}
		K(R):=\{z\in \mathcal{X}:  \lVert\phi(z)\rVert^{2}_{H^{1}(\Omega)}\leq R\},
	\end{align*}
	for any $R>0$. We observe that $\phi(z)$ is bounded in $H^{1}(\Omega)$ for $z\in K(R)$, and therefore relatively compact in $L^{1}(\Omega)$. This implies that any sequence $\{w_{k}\}_{k=1}^{\infty}$ in $K(R)$ has a subsequence, which we still denote by $w_{k}$, such that 
	\begin{equation*}
		\phi(w_{k})\to w\ \text{in}\ L^{1}(\Omega).
	\end{equation*}
Then we have that $\phi(w_{k})$ converges to $w$ a.e. in $\Omega$ along some subsequence, denoted again by $\phi(w_{k})$. Since the condition \eqref{eq:P1} implies that $\phi^{-1}$ is continuous on $\phi(I)$, we get that $w_{k}=\phi^{-1}(\phi(w_{k}))$ converges to $\phi^{-1}(w)$ a.e. in $\Omega$. Recalling that $|w_{k}|< 1$ for all $k$, we use Lebesgue's dominated convergence theorem to conclude that
    \begin{equation*}
		w_{k}\to \phi^{-1}(w)\ \text{in}\ L^{1}(\Omega), 
	\end{equation*}which implies that the set $K(R)$ is relatively compact in $L^{1}(\Omega)$. 
	Setting $K:=K(C/\xi^{\kappa})$, we observe that \eqref{eq:Es1} implies $\mathcal{S}(\xi)u_{0}=u(\xi)\in K$. Moreover, taking 
	\begin{align*}        \mathcal{K}:=\bigcup_{t\in[0,T]}\mathcal{S} (t)K, 
	\end{align*}
	we have $\mathcal{S}(T)u_{0}=\mathcal{S}(T-\xi)\mathcal{S}(\xi)u_{0}\in \mathcal{K}$. By the definition of $\mathcal{K}$, we observe that  $\mathcal{K}=Z([0,T]\times K)$, where $Z(t,y):=\mathcal{S}(t)y$. By the relative compactness of $K$, it follows that $[0,T]\times K$ is also relatively compact. Furthermore, the function $Z$ is continuous in $L^1(\Omega)$. Indeed, for any sequences $\{y_{k}\}_{k=1}^{\infty}\subset L^{1}(\Omega)$ and $\{t_{k}\}_{k=1}^{\infty}\subset [0,T]$ such that $y_{k}\to y$ and $t_{k}\to t$ as $k\to \infty$, by \eqref{eq:L1C}, we get
	\begin{align*}
		\lVert\mathcal{S}(t_{k})y_{k}-\mathcal{S}(t)y\rVert_{L^{1}(\Omega)}\leq\lVert\mathcal{S}(t_{k})y_{k}-\mathcal{S}(t_{k})y\rVert_{L^{1}(\Omega)}+\lVert\mathcal{S}(t_{k})y-\mathcal{S}(t)y\rVert_{L^{1}(\Omega)}\\\leq  e^{Lt_{k}}\lVert y_{k}-y\rVert_{L^{1}(\Omega)}+\lVert\mathcal{S}(t_{k})y-\mathcal{S}(t)y\rVert_{L^{1}(\Omega)}.
	\end{align*}
	The first term on the right-hand side converges to zero as $k\to \infty$ and the second term as well, due to the Lipschitz continuity of $\mathcal{S}$ with respect to $t$ in $L^{1}(\Omega)$, see Theorem~\ref{WPS}. Consequently, the set $\mathcal{K}$ is relatively compact in $L^{1}(\Omega)$. Therefore, $u(T)=\mathcal{S}(T)u_{0}$ has a convergent subsequence which converges to $u$. Hence, considering that $C$ is independent of $u_{0}$, the set $\mathcal{S}(T) \mathcal{X}$ is relatively compact in $L^{1}(\Omega)$.
	Since $\mathcal{X}$ is positively invariant, we have 
	\begin{equation*}
		\mathcal{S}(t)\mathcal{X}=\mathcal{S}(T)\mathcal{S}(t-T)\mathcal{X}\subset \mathcal{S}(T)\mathcal{X}
	\end{equation*}
	for $t\geq T$, which completes the proof. 
\end{proof}

Using Lemma~\ref{ACS}, we are in a position to prove Theorem~\ref{GAS}.

\begin{proof}[Proof of Theorem~\ref{GAS}:] 
	By Theorem~\ref{WPS}, the semigroup $\{\mathcal{S}(t)\}_{t\geq 0}$ is continuous in $L^1(\Omega)$ with respect to $t$. Therefore, using that $\mathcal{X}$ is a bounded absorbing set together with Lemma~\ref{ACS}, we can apply the abstract theorem \cite[ Theorem 1.1]{Temam} on the existence of
	global attractors to conclude the proof of the theorem. 
\end{proof}

Finally, we present a sketch of the proof of the existence of the global attractor  to the coupled problem \eqref{eq:CP}. As for the scalar problem \eqref{eq:SP}, we can consider smooth non-degenerate approximations and derive uniform estimates for the approximate solutions. The difficulties lie in the estimates for the degenerate equation which we provided in this section, the arguments for the semilinear equation are standard, see \cite{Efendiev2009}. 
The semigroup in the phase space $\mathcal{Y}$ is then obtained analogously as in the scalar case, see also \cite{Efendiev2009}. 
To prove the asymptotically compactness of the semigroup we use arguments from \cite{Feireisl}.

\begin{proof}[Proof of Theorem~\ref{GAC}:] Analogously as in the proof of Theorem~\ref{GAS}, we obtain the continuity of the semigroup $\{\mathcal{T}(t)\}_{t\geq 0}$ in $L^1(\Omega)\times L^1(\Omega)$ with respect to $t$ and the existence of a bounded absorbing set, which is $\mathcal{Y}$ by Theorem~\ref{WPC}. Therefore, it remains to show the asymptotic compactness of the semigroup. We only present a sketch of the proof as the arguments are the same as in the proof of Lemma~\ref{ACS}. 
	
	As in the proof of Lemma~\ref{ACS}, it can be shown that the set
	\begin{align*}
		K(R):=\{(z_{1},z_{2})\in \mathcal{Y}: ||\phi(z_{1})||_{H^{1}(\Omega)}+|| z_{2}||_{H^{1}(\Omega)} \leq R\},
	\end{align*}for any $R>0$, is relatively compact in $L^{1}(\Omega)\times L^{1}(\Omega)$. Moreover, by an analogous estimate as  \eqref{eq:fphi2} from Theorem~\ref{thm:limit} for the scalar problem, for any $T>0$, there exists $\xi\in[0,T]$ such that $\mathcal{T}(\xi)(u_{0},v_{0})=(u(\xi),v(\xi))\in K$, where $K=K(C)$ for some constant $C>0$, which is independent of $(u_{0},v_{0})$.
	Hence, it follows that $\mathcal{T}(T)(u_{0},v_{0})\in \mathcal{K}$, for any $(u_{0},v_{0})\in \mathcal{Y}$, where the set 
	\begin{align*}   		\mathcal{K}=\bigcup_{t\in[0,T]}\mathcal{T}(t)K
	\end{align*}
	is relatively compact in $L^{1}(\Omega)\times L^{1}(\Omega)$. Therefore, as in the proof of Lemma~\ref{ACS} we conclude that the set $\bigcup\limits_{t\geq T}\mathcal{T}(t)\mathcal{Y}$ is relatively compact. The abstract theorem \cite[Theorem 1.1]{Temam} on the existence of the global attractor can be applied, which completes the proof. 
\end{proof}

\section{Existence of Exponential Attractors}\label{sec:EA}

We first consider the scalar problem \eqref{eq:SP} and show the existence of an exponential attractor for the corresponding semigroup $\{\mathcal{S}(t)\}_{t\geq 0}$, i.e. we prove  Theorem~\ref{EAS}. To this end, we follow and adapt the arguments in \cite[Theorems 2.1 and 3.1]{EfendievZelik2008} to our case. However, different from \cite{EfendievZelik2008}, where first the finite fractal dimension of the global attractor was shown and then the exponential attractor was constructed from scratch, we apply an abstract existence result for exponential attractors for discrete time semigroups.  
Subsequently, the exponential attractor for the continuous time semigroup can be constructed in a standard manner, see \cite{CzajaEfendiev, EfendievZelik2008, CzajaSonner, SonnerPhD}. This significantly shortens the existence proof for the exponential attractor, and the finite fractal dimension of the global attractor is an immediate consequence. 
More specifically, the existence result for discrete time exponential attractors we will use is the following proposition which is a slight generalization of \cite[Theorem 2.1]{CzajaEfendiev}, see also \cite[Remark 3.2]{EfendievZelik2008}. Here, we give a specific estimate of the fractal dimension of the exponential attractor in terms of the constants in the relevant estimates.

\begin{proposition}\label{EfendievZelik2008}
	Let $\{S(n)\}_{n\in\mathbb{N}_0}$ be a discrete semigroup in a Banach space $E$ and $K$ be a compact subset of $E$ such that $S(n)K\subset K$ for all $n\in\mathbb{N}_0$. Assume that for every $u_*\in K$ and $0<\varepsilon\leq \varepsilon_{0}$, there exists a pair of Banach spaces $X_{u_*,\varepsilon}$ and $Y_{u_*,\varepsilon}$ such that $X_{u_*,\varepsilon}$ is compactly embedded into $Y_{u_*,\varepsilon}$ and this embedding is uniform in the sense of the Kolmogorov $\varepsilon$-entropy, i.e.  for $\rho>0$
	$$
	N_{\rho}^{Y_{u_*,\varepsilon}}(B^{X_{u_*,\varepsilon}}_1(0))\leq \widetilde N_\rho,
	$$
	where $\widetilde N_\rho$ is independent of $u_*$ and $\varepsilon\leq \varepsilon_0.$
	Moreover, assume that there exists a map $\mathcal{T}_{u_*,\varepsilon}:B^{E}_{\varepsilon}(u_*)\cap K\to X_{u_*,\varepsilon}$ such that for any $u_{1}, u_{2}\in B^{E}_{\varepsilon}(u_*)\cap K$ we have
	\begin{align*}
		\lVert \mathcal{T}_{u_*,\varepsilon}u_{1}-\mathcal{T}_{u_*,\varepsilon}u_{2}\rVert_{X_{u_*,\varepsilon}}&\leq \kappa\lVert u_{1}-u_{2}\rVert_{E},\\
		\lVert S(1)u_{1}-S(1)u_{2}\rVert_{E}&\leq \eta\lVert u_{1}-u_{2}\rVert_{E}+\mu\lVert \mathcal{T}_{u_*,\varepsilon}u_{1}-\mathcal{T}_{u_*,\varepsilon}u_{2}\rVert_{Y_{u_*,\varepsilon}},
	\end{align*}
	where the constants $\kappa>0, \eta\in[0,\frac{1}{2})$ and $\mu>0$ are independent of $u_*$ and $\varepsilon\leq \varepsilon_0.$
	Then, for any $\nu\in (0,\frac{1}{2}-\eta)$ there exists an exponential attractor $\mathcal{M}_{\nu}\subset K$ for the semigroup $\{S(n)\}_{n\in\mathbb{N}_0}$ and its fractal dimension is bounded by 
	\begin{align*}
		\text{dim}_{f}^{E}(\mathcal{M}_\nu)\leq \log_{\frac{1}{2(\eta+\nu)}}\widetilde N_{\frac{\nu}{\kappa\mu}}.
	\end{align*}  
\end{proposition}

The proof of Theorem~\ref{EAS} is based on several lemmas that establish the required estimates in Proposition \ref{EfendievZelik2008} for suitable spaces $E, X_{u_*,\varepsilon}$ and $Y_{u_*,\varepsilon}$. 

\begin{remark}
	All estimates obtained from now on can be  justified by approximating the solutions of \eqref{eq:SP} and \eqref{eq:CP} by smooth solutions of non-degenerate problems, see Section~\ref{sec:GA}. 
\end{remark}

First of all, we introduce the set 
\begin{align*}
	\mathcal{C}:=\overline{\bigcup_{t\geq 1}\mathcal{S}(t)\mathcal{X}},
\end{align*}
where the closure is taken in $C^{\alpha}(\Bar{\Omega})$. By \eqref{eq:HC}, we have $\lVert \mathcal{C}\rVert_{C^{\alpha}(\Bar{\Omega})}\leq M$ and due to the construction of $\mathcal{C}$, we have $\mathcal{S}(t)\mathcal{C}\subset \mathcal{C}$ for all $t\geq 0$. Let $u_{0}\in \mathcal{C}, u_{0} \neq 0$ be arbitrary. We introduce the sets 
\begin{align}
	\begin{split}
		L(\theta)=L(\theta,u_{0}):=\{x\in \Omega:|u_{0}(x)|>\theta\},\\
		S(\theta)=S(\theta,u_{0}):=\{x\in \Omega:|u_{0}(x)|<\theta\},
	\end{split}
	\label{eq:LS}
\end{align}for every $\theta>0$, which clearly satisfy $S(\theta_{1})\subset S(\theta_{2})$ and $L(\theta_{2})\subset L(\theta_{1})$ if $\theta_{1}\leq \theta_{2}$. Moreover, by \eqref{eq:HC}, we observe that these sets are open and satisfy\begin{align*}
	\partial S(\theta)=\partial L(\theta)=\{ x\in\Omega: |u_{0}(x)|=\theta\}, \quad   \Omega=S(\theta)\cup L(\theta)\cup \partial L(\theta).
\end{align*}

Furthermore, choosing $\delta>0$ and $\theta>0$ small enough such that the sets $\partial S(\theta)$ and $\partial S(\theta+\delta)$ are nonempty, we observe that \begin{equation*}
	\lVert x-y\rVert\geq \frac{1}{M^{1/\alpha}}|u_{0}(x)-u_{0}(y)|^{1/\alpha}\geq 
	\frac{1}{M^{1/\alpha}}\big|\lvert u_{0}(x)\rvert -\lvert u_{0}(y)\rvert\big| ^{1/\alpha}
	=\left(\frac{\delta}{M}\right)^{1/\alpha},
\end{equation*}
for every $x\in \partial S(\theta+\delta)$, $y\in \partial S(\theta)$ and some constant $M>0$. Therefore, we get\begin{equation}
	d(\partial S(\theta+\delta),\partial S(\theta))\geq C_{\delta},\label{eq:posdistsets}
\end{equation}
where $ C_{\delta}\equiv \left(\frac{\delta}{M}\right)^{1/\alpha}$ and $d(A,B)=\inf\limits_{x\in A}\inf\limits_{y\in B}\lVert x-y\rVert$ denotes the standard metric distance between two sets $A$ and $B$ in $\mathbb{R}^{n}$.

\begin{lemma}\label{sign}
	Let the assumptions of Theorem \ref{EAS} hold. There exists $T>0$ and $\varepsilon>0$ such that for any solution $v$ of Problem \eqref{eq:SP} with initial data $v_{0}\in B^{L^{1}(\Omega)}_{\varepsilon}(u_{0})\cap \mathcal{C}$ 
	we have
	\begin{align}
		\begin{split}
			f_{u}(x,v(t,x))&<-\beta\quad \text{for} \ x\in S(4\theta),\ t\in [0,T],\\
			|v(t,x)|&>\theta/4\quad \text{for}\ x\in L(\theta),\ t\in [0,T],
		\end{split}
		\label{eq:ressigncond3}
	\end{align}
	uniformly with respect to $v_{0}$.
\end{lemma}

\begin{proof}
	Due to the \eqref{eq:signcondf} and \eqref{eq:HC}, we can fix $\theta, \beta>0$ such that 
	\begin{align}
		f_{u}(\cdot,u)<-3\beta<0 \quad \text{in}\ \Omega\quad \forall |u|<5\theta. \label{eq:ressigncond}
	\end{align}
	Applying Proposition A.2 in \cite{EfendievZelik2008} with $V=S(4\theta)$, by virtue of \eqref{eq:posdistsets} with sufficiently small $\delta>0$, we can choose a cut-off function $\psi\in C^{\infty}(\mathbb{R}^{n})$ such that $\psi\geq 0$, \begin{equation}
		\psi(x)=\left\{ 
		\begin{array}{c}
			1, \ x\in S(4\theta), \\ 
			0,\ x\in L(5\theta),%
		\end{array}%
		\right. \label{eq:cutoffdef}
	\end{equation}
	and
	\begin{equation}
		\lVert\psi\rVert_{C^{k}(\overline{\Omega})}\leq C_{k},\ k\in \mathbb{N},\label{eq:cutoffbdd}
	\end{equation}
	where the constants $C_{k}\geq 0$ only depend on $M,\alpha$ and $k$, and are independent of $u_{0}$.
	
	Recall that by \eqref{eq:HC}, the solution $u(t):=\mathcal{S}(t)u_{0}$ belongs to $C^{\alpha}([0,T]\times \overline{\Omega})$. Hence, for any $T>0$ we have 
	\begin{equation*}
		|u(t,x)|\leq |u(t,x)-u_{0}(x)|+|u_{0}(x)|< MT^{\alpha}+5\theta\quad \text{for}\ x\in S(5\theta),\ t\in [0,T],
	\end{equation*}
	which, together with \eqref{eq:ressigncond} and the continuity of $f_{u}$, implies that
	\begin{equation*}
		f_{u}(x,u(t,x))<-3\beta+K\qquad \text{for}\  x\in S(5\theta),\ t\in [0,T],
	\end{equation*}
	for some constant $K$ depending only on $M,\alpha$ and $T$. Moreover, $K\leq \beta$ if  $T>0$ is sufficiently small. On the other hand, for $x\in L(\theta)$, we have
	\begin{equation*}
		|u(t,x)|\geq -|u(t,x)-u_{0}(x)|+|u_{0}(x)|>-MT^{\alpha}+\theta\qquad \text{for}\ t\in [0,T].
	\end{equation*}
	Therefore, choosing $T$ sufficiently small such that $K\leq\beta$ and 
	$MT^{\alpha}<\theta/2$, we obtain
	\begin{align}
		\begin{split}
			f_{u}(x,u(t,x))&<-2\beta\qquad \text{for}\ x\in S(5\theta),\ t\in [0,T],\\
			|u(t,x)|&>\theta/2\qquad\  \text{for}\ x\in L(\theta),\ t\in [0,T].
		\end{split}
		\label{eq:ressigncond2}
	\end{align}
	Note that $T>0$ depends only on $M, \alpha, \beta$ and $\theta$.
	On the other hand, let $v$ be a solution with initial data $v_{0}\in B^{L^{1}(\Omega)}_{\varepsilon}(u_{0})\cap \mathcal{C}, \varepsilon>0$. Then, using the following interpolation inequality (see \cite{EfendievZelik2008})
	\begin{equation}\label{eq:interpol}
		\lVert \omega\rVert_{C(\overline{\Omega})}\leq C_0\lVert \omega\rVert^{\gamma}_{L^{1}(\Omega)}\lVert \omega\rVert^{1-\gamma}_{C^{\alpha}(\Bar{\Omega})}, \qquad \omega \in C^{\alpha}(\Bar{\Omega}),
	\end{equation}
	for some $\gamma\in (0,1)$ and $C_0\geq0$, together with \eqref{eq:L1C} and \eqref{eq:HC}, we obtain 
	for $ x\in \Omega,\ t\in[0,T]$
	\begin{align*}
		|v(t,x)-u(t,x)|\leq \lVert v(t)-u(t)\rVert_{C(\overline{\Omega})}&\leq C_0\lVert v(t)-u(t)\rVert^{\gamma}_{L^{1}(\Omega)}\lVert v(t)-u(t)\rVert^{1-\gamma}_{C^{\alpha}(\Bar{\Omega})}\\
		&\leq C_0(2M)^{1-\gamma}e^{\gamma LT}\lVert v_{0}-u_{0}\rVert^{\gamma}_{L^{1}(\Omega)}\leq C\varepsilon^{\gamma}.
	\end{align*}
	Then, by $\eqref{eq:ressigncond2}_{2}$, it follows that 
	\begin{equation*}
		|v(t,x)|\geq |u(t,x)|-|v(t,x)-u(t,x)|>\theta/2-C\varepsilon^{\gamma}\quad \text{for}\ x\in L(\theta),\ t\in[0,T].
	\end{equation*} 
	Hence, for all sufficiently small $\varepsilon$ independent of $u_{0}\in \mathcal{C}$, considering $\eqref{eq:ressigncond2}_{1}$, we obtain \eqref{eq:ressigncond3}.
\end{proof}

Next, we show a smoothing property for the difference of two solutions of \eqref{eq:SP}. We first consider the domain $L(\theta)$, where the equation $\eqref{eq:SP}_{1}$ is non-degenerate. The following lemma is analogous to the standard interior regularity estimates for linear parabolic problems.

\begin{lemma}\label{smoothinglemma}
	Let the assumptions of Theorem \ref{EAS} hold and
	$u_{1},u_{2}$ be two solutions of \eqref{eq:SP} with initial data $u_{i}(0)\in B^{L^{1}(\Omega)}_{\varepsilon}(u_{0})\cap \mathcal{C}, i=1,2$. Then, for any $t_0\in(0,T)$ we have
	\begin{equation}
		\lVert u_{1}-u_{2}\rVert_{C^{\alpha}([t_{0},T]\times \overline{L(3\theta)})}\leq C_{t_{0}}\lVert u_{1}(0)-u_{2}(0)\rVert_{L^{1}(\Omega)},\label{eq:smoothing}
	\end{equation}
	where $T>0$ and $\varepsilon>0$ are as in Lemma~\ref{sign} and the constant $C_{t_{0}}>0$ depends on $t_{0}$, but is independent of $\varepsilon, u_0$ and $u_1,u_2$. 
\end{lemma}

\begin{proof}    
	We first show that for any $r>0$ and $t_0\in(0,T)$ the functions $u_{1}$ and $u_{2}$ satisfy 
	\begin{equation}
		\lVert u_{i}\rVert_{W^{(1,2),r}((t_{0}/2,T)\times L(2\theta))}\leq C_{r,t_0},\qquad  i=1,2, \label{eq:anis.es.ui}
	\end{equation}
	where $W^{(1,2),r}(D)$ denotes the anisotropic Sobolev space consisting of functions whose $t$-derivatives up to order $1$ and $x$-derivatives up to order $2$ belong to $L^{r}(D)$. The constant $C_{r,t_0}$ depends on $r$ and $t_0$, but is independent of $u_{0},\varepsilon$ and the trajectories $u_{1}$ and $u_{2}$. It suffices to verify \eqref{eq:anis.es.ui} for $u=u_{1}$ since the proof is analogous for $u=u_{2}$. To this end, we introduce $v(t,x):=\phi(u(t,x))$. Then, considering $\eqref{eq:ressigncond3}_{2}$ and that $\phi(u)$ is non-degenerate if $|u|>\theta/4>0$, the function $v$ solves the equation
	\begin{equation}
		v_{t}=a\Delta v+af(\cdot,u)\qquad \text{in}\ (0,T]\times L(\theta),\label{eq:veq}
	\end{equation}
	where $a(t,x):=\phi'(u(t,x))$. Furthermore, by \eqref{eq:HC}, the continuity of $\phi'$ and $f$, the function $a$ is uniformly Hölder continuous with respect to $t$ and $x$ and the function $af(\cdot,u)$ is uniformly bounded in $L^{\infty}(\Omega_{T})$. Additionally, due to \eqref{eq:P3} and $\eqref{eq:ressigncond3}_{2}$, we have \begin{equation}
		a(t,x)\geq C \qquad \forall (t,x)\in [0,T]\times L(\theta),\label{eq:lowerboundofa}
	\end{equation}
	where the constant $C$ is also independent of $u_{0}$ and $u$. Hence,  we can use an $L^{r}$-interior regularity estimate for the solution of the linear non-degenerate equation \eqref{eq:veq} (see \cite[Proposition A.4]{EfendievZelik2008}), and obtain, by \eqref{eq:posdistsets}, for any $r>2$ and $t_0\in(0,T)$,
	\begin{equation}
		\lVert v\rVert_{W^{(1,2),r}((t_{0}/2,T)\times L(2\theta))}\leq C_{r, t_0}(\rVert af(\cdot,u)\rVert_{L^{r}((0,T)\times L(\theta))}+\lVert v\rVert_{L^{1}((0,T)\times L(\theta))})\leq C'_{r,t_0},\label{eq:anis.es.v}
	\end{equation}
	for some constant $C_{r,t_0}'\geq 0.$
	Then, taking into account that $u = \phi^{-1}(v)$, the fact that $\phi\in C^{2}(I)$ and non-degenerate away from zero, we find, by \eqref{eq:lowerboundofa}-\eqref{eq:anis.es.v}, that
	\begin{align*}
		&\lVert u\rVert_{W^{(1,2),r}((t_{0}/2,T)\times L(2\theta))}^r=\lVert\phi^{-1}(v)\rVert_{W^{(1,2),r}((t_{0}/2,T)\times L(2\theta))}^r\\ 
		&= \lVert\phi^{-1}(v)\rVert_{L^{r}((t_{0}/2,T)\times L(2\theta))}^r+\left\Vert\frac{v_{t}}{\phi'(u)}\right\Vert_{L^{r}((t_{0}/2,T)\times L(2\theta))}^r
		+\sum_{i=1}^n\left\Vert\frac{v_{x_i}}{\phi'(u)}\right\Vert_{L^{r}((t_{0}/2,T)\times L(2\theta))}^r\\&+\sum_{i,j=1}^n\left(\left\Vert\frac{\phi''(u)v_{x_i}v_{x_j}}{(\phi'(u))^{2}}\right\Vert_{L^{r}((t_{0}/2,T)\times L(2\theta))}^r+\left\Vert\frac{v_{x_ix_j}}{\phi'(u)}\right\Vert_{L^{r}((t_{0}/2,T)\times L(2\theta))}^r\right)\\
		&\leq C\left(1+\lVert v\rVert^r_{W^{(1,2),r}((t_{0}/2,T)\times L(2\theta))}+\lVert\nabla v\rVert^{2r}_{L^{2r}((t_{0}/2,T)\times L(2\theta))}\right)\leq C_{r,t_0}'',
	\end{align*}
	for some constant $C_{r}''>0$, which yields the estimate \eqref{eq:anis.es.ui} for $u=u_{1}$.
	
	To prove estimate \eqref{eq:smoothing}, we consider $w=u_{1}-u_{2}$ in $[t_0/2,T]\times L(2\theta)$ which satisfies
	\begin{align}
		w_{t}&=\Delta(l_{1}(t)w)+l_{2}(t)w\qquad \text{in}\ (t_0/2,T]\times L(2\theta),
		\label{eq:differenceprob}
	\end{align}
	where\begin{align}
		\begin{split}
			l_{1}(t)&:=\int_{0}^{1}\phi'(su_{1}(t)+(1-s)u_{2}(t))ds, \\
			l_{2}(t)&:=\int_{0}^{1}f_{u}(\cdot,su_{1}(t)+(1-s)u_{2}(t))ds.
		\end{split}
		\label{eq:l1l2def}
	\end{align}
	First, we use that $\phi\in C^{2}(I)$ and \eqref{eq:HC} to conclude that $l_{1}$ satisfies
	\begin{equation}
		||l_{1}||_{C^{\alpha}([0,T]\times\Bar{\Omega})}\leq C,\label{eq:l1es}
	\end{equation}
	where  $C$ is independent of $u_{1}$ and $u_{2}$. Furthermore, by \eqref{eq:anis.es.ui}, we have
	\begin{equation}
		||\partial_{t}l_{1}||_{L^{r}((t_{0}/2,T)\times L(2\theta))}\leq C\sum_{i=1}^{2}||\partial_{t}u_{i}||_{L^{r}((t_{0}/2,T)\times L(2\theta))}\leq C_{r,t_0}''',\label{eq:timederl1}
	\end{equation}
	for some constant $C_{r}'''>0$ independent of $u_{0}, u_{1}$ and $u_{2}$.  Also, by $\eqref{eq:ressigncond3}_{2}$ and \eqref{eq:P3}, we have
	\begin{equation}
		l_{1}(t,x)\geq C
		>0 \qquad \text{for}\ (t,x)\in [t_{0}/2,T]\times L(2\theta),\label{eq:lowerboundl1}
	\end{equation}
	where the constant $C>0$ is independent of $u_{0}, u_{1}$ and $u_{2}$.
	
	Now we introduce $Z= l_{1}w$ which, by $\eqref{eq:differenceprob}$, solves
	\begin{equation}
		\partial_{t}Z=b(t,x)\Delta Z+l(t,x)Z\qquad \text{in}\ (t_{0}/2,T]\times L(2\theta),\label{eq:eqZ}
	\end{equation}
	where \begin{align*}
		b(t,x):=l_{1}(t,x)\ \ \text{and} \ \ l(t,x):=l_{2}(t,x)+\frac{\partial_{t}l_{1}(t,x)}{l_{1}(t,x)}.
	\end{align*}
	We observe that $l_{2}$ is uniformly bounded in $L^{\infty}(\Omega_{T})$. Hence, considering \eqref{eq:l1es}-\eqref{eq:lowerboundl1} we now apply the $L^{q}$-interior regularity estimate (see \cite{EfendievZelik2008}, Proposition A.4 and Corollary A.1) for \eqref{eq:eqZ}, and conclude that for any $q>2$
	\begin{equation}
		||Z||_{W^{(1,2),q}((t_{0},T)\times L(3\theta))}\leq C_{q,t_0}||Z||_{L^{1}((t_{0}/2,T)\times L(2\theta))}\leq C_{q,t_0}'||w||_{L^{1}(\Omega_{T})},\label{eq:anis.es.Z}
	\end{equation}
	If we fix $q$ so large such that the embedding $W^{(1,2),q}((t_{0},T)\times L(3\theta))\hookrightarrow C^{\alpha}([t_{0},T]\times \overline{L(3\theta)})$ is continuous, it follows from the definition of $Z$, \eqref{eq:l1es}, \eqref{eq:lowerboundl1} and \eqref{eq:anis.es.Z} that
	\begin{align*}
		||w||_{C^{\alpha}([t_{0},T]\times \overline{L(3\theta)})}
		=\left\Vert \frac{Z}{l_{1}}\right\Vert_{C^{\alpha}([t_{0},T]\times\overline{L(3\theta)})}
		&\leq C||Z||_{C^{\alpha}([t_{0},T]\times \overline{L(3\theta)})}
		\\ & \leq C||Z||_{W^{(1,2),q}((t_{0},T)\times L(3\theta))}
		\\ & \leq C_{q,t_0}||w||_{L^{1}(\Omega_{T})}\leq C_{q,t_0}e^{LT}||w(0)||_{L^{1}(\Omega)},
	\end{align*}
	which implies \eqref{eq:smoothing}.
\end{proof}

The following lemma yields a contraction property for the difference of two solutions. 
\begin{lemma}\label{contractionlemma}
	Let the assumptions of Theorem \ref{EAS} hold and $T>0$ and $\varepsilon>0$ be as in Lemma~\ref{sign}.
	For any two solutions $u_{1}, u_{2}$  of \eqref{eq:SP} with initial data $u_{i}(0)\in B^{L^{1}(\Omega)}_{\varepsilon}(u_{0})\cap \mathcal{C}, i=1,2$, there exists $t_*\in (0,T)$ such that the estimate
	\begin{equation}
		||u_{1}(T)-u_{2}(T)||_{L^{1}(S(4\theta))}\leq \frac{1-\delta}{2}||u_{1}(0)-u_{2}(0)||_{L^{1}(\Omega)}+C||u_{1}-u_{2}||_{L^{1}([t^*,T]\times L(4\theta))}\label{eq:contract}
	\end{equation}
	holds, where $\delta\in (0,1)$
	and $C>0$ are independent of $\varepsilon$, $u_{0}$ and $u_1,u_2$.     
\end{lemma}

\begin{proof}
	We set $w=u_1-u_2$ and consider \eqref{eq:differenceprob} on the set $(0,T]\times S(4\theta)$. 
	Multiplying the equation by
	\begin{equation}
		\psi(x)\text{sgn}(w(t,x))=\psi(x)\text{sgn}(\psi(x)l_{1}(t,x)w(t,x)),\label{eq:3.30}
	\end{equation}
	where $\psi$ is the cut-off function defined in \eqref{eq:cutoffdef}, and integrating over $\Omega$ we get 
	\begin{align*}
		\left<w_{t},\psi \text{sgn}(w)\right>=\left<\Delta(l_{1}(t)w),\psi \text{sgn}(\psi l_{1}(t)w)\right>+\left<l_{2}(t)w,\psi sgn(w)\right>.
	\end{align*}
	Therefore, using the relation\begin{equation*}
		\psi\Delta(l_{1}(t)w)=\Delta(\psi l_{1}(t)w)-2\nabla\psi\cdot\nabla(l_{1}(t)w)- l_{1}(t)w\Delta\psi,
	\end{equation*}we have\begin{align*}
		\partial_{t}\left< \psi,|w| \right> &=\left<\Delta(\psi l_{1}(t)w),\text{sgn}(\psi l_{1}(t)w)\right>-2\left<\nabla\psi\cdot\nabla(l_{1}(t)w),\text{sgn}(\psi l_{1}(t)w)\right>\notag 
		\\&\quad -\left<l_{1}(t)w\Delta\psi,\text{sgn}(\psi l_{1}(t)w)\right>+\left<l_{2}(t)w,\psi \text{sgn}(w)\right>.
	\end{align*}
    Then, using the inequality $\langle\Delta v, \text{sgn}(v)\rangle\leq 0$ we obtain
    \begin{equation*}
		\partial_{t}\left<\psi,|w|\right>\leq\left<\Delta\psi,l_{1}(t)|w|\right>+\left<\psi,l_{2}(t)|w|\right>+2\left< l_{1}(t)w\nabla\psi,\nabla(\text{sgn}(w))\right>,
	\end{equation*}
    which can be justified by considering a smooth approximating of the $\text{sgn}$ function, performing integration by parts and passing to the limit.
    Therefore, using $\eqref{eq:ressigncond3}_{1}$, \eqref{eq:cutoffbdd} and \eqref{eq:l1es}, and the fact that $\Delta\psi(x)=|\nabla\psi(x)|=0$ for $x\in S(4\theta)$, we conclude that 
	\begin{align*}		&\partial_{t}\left<\psi,|w(t)|\right>+\beta\left<\psi,|w(t)|\right>\\
    \leq& \int\limits_{L(4\theta)}\Delta\psi(x)l_{1}(t,x)|w(t,x)|dx+\int\limits_{L(4\theta)}|\nabla\psi(x)|l_{1}(t,x)|w(t,x)|dx\leq C||w(t)||_{L^{1}(L(4\theta))}.
	\end{align*}
	Consequently, applying Gronwall's lemma and using that $\psi(x)=1$ for $x\in S(4\theta)$, it follows that
	\begin{equation*}
		||w(T)||_{L^{1}(S(4\theta))}\leq e^{-\beta(T-t_{*})}||w(t_{*})||_{L^{1}(L(4\theta))}+C\int_{t_{*}}^{T}e^{-\beta(T-s)}||w(s)||_{L^{1}(L(4\theta))}ds,
	\end{equation*}
	for an arbitrary $t_{*}\in(0,T)$. Therefore, together with \eqref{eq:L1C}, we obtain
	\begin{equation*}
		||u_{1}(T)-u_{2}(T)||_{L^{1}(S(4\theta))}\leq e^{Lt_{*}-\beta(T-t_{*})}||u_{1}(0)-u_{2}(0)||_{L^{1}(\Omega)}+C_{t_{*}}||u_{1}-u_{2}||_{L^{1}([t_{*},T]\times L(4\theta))},
	\end{equation*}
	for some constant $C_{t_{*}}$. By fixing $t_{*}$ sufficiently small such that $e^{Lt_{*}-\beta(T-t_{*})}< (1-\delta)/2<1$, we obtain \eqref{eq:contract}.
\end{proof}

Finally, we use Proposition \ref{EfendievZelik2008} and the results obtained in the Lemmas~\ref{smoothinglemma}-~\ref{contractionlemma}  to establish the existence of an exponential attractor. 

\begin{proof}[Proof of Theorem~\ref{EAS}:] 
	Let $u_0\in\mathcal{C}$, $T>0$ and $\varepsilon>0$ be as in Lemma~\ref{sign} and  $u_{1}, u_{2}$ be two solutions of \eqref{eq:SP} with initial data $u_{i}(0)\in B^{L^{1}(\Omega)}_{\varepsilon}(u_{0})\cap \mathcal{C}, i=1,2$. 
	We take $t_*\in(0,T)$ as in Lemma \ref{contractionlemma}. 
	By \eqref{eq:smoothing}, we have the estimate 
	\begin{align}
		\lVert u_{1}-u_{2}\rVert_{C^{\alpha}([t_{*},T]\times L(3\theta))}&\leq \kappa\lVert u_{1}(0)-u_{2}(0)\rVert_{L^{1}(\Omega)},\label{eq:smoothpart}
	\end{align}
	where $\kappa=C_{t_*}.$
	Moreover, taking into account the following facts 
	\begin{align*}
		\lVert w(T)\rVert_{L^{1}(\Omega)}&\leq \lVert w(T)\rVert_{L^{1}(S(4\theta))}+\lVert w(T)\rVert_{L^{1}(L(7\theta/2))},\\
		\lVert w(T)\rVert_{L^{1}(L(7\theta/2))}&
		\leq C\lVert w\rVert_{C([t_{*},T]\times L(7\theta/2))},
	\end{align*}
	and using \eqref{eq:contract}, we deduce the estimate\begin{align}
		\lVert u_{1}(T)-u_{2}(T)\Vert_{L^{1}(\Omega)}&
		\leq  \lVert u_{1}(T)-u_{2}(T)\Vert_{{L^{1}(S(4\theta))}}+ \lVert u_{1}(T)-u_{2}(T)\Vert_{L^{1}(L(7\theta/2))}\notag\\
		&\leq \frac{1-\delta}{2}\lVert u_{1}(0)-u_{2}(0)\rVert_{L^{1}(\Omega)}+C||u_{1}-u_{2}||_{L^{1}([t_{*},T]\times L(4\theta))}\notag\\
		&\quad+C\lVert u_{1}-u_{2}\rVert_{C([t_{*},T]\times L(7\theta/2))}\notag\\&\leq \frac{1-\delta}{2}\lVert u_{1}(0)-u_{2}(0)\rVert_{L^{1}(\Omega)}+ \mu\lVert u_{1}-u_{2}\rVert_{C([t_{*},T]\times L(7\theta/2))},\label{eq:contpart}
	\end{align}
	for some constant $\mu>0$. We note that $T > 0, \delta\in(0,1), \kappa>0$ and $\mu>0$ are independent of $\varepsilon\leq \varepsilon_{0}$, $u_{0}$ and $u_{i}, i=1,2$, and that the embedding\begin{equation*}
		C^{\alpha}([t_{*},T]\times \overline{L(3\theta)})\hookrightarrow C([t_{*},T]\times \overline{L(7\theta/2)})
	\end{equation*} 
	is compact and uniform in the sense of the Kolmogorov $\varepsilon$-entropy. 
	Hence, we can apply Proposition  \ref{EfendievZelik2008} to the discrete time semigroup 
	$S(n)=\mathcal{S}(Tn), n\in\mathbb{N}_0$, with $E=L^{1}(\Omega),$ $K=\mathcal{C}, u_{*}=u_{0},$ $X_{u_{*},\varepsilon}=C^{\alpha}([t_{*},T]\times \overline{L(3\theta)})$ and $Y_{u_{*},\varepsilon}=C([t_{*},T]\times \overline{L(7\theta/2)})$. Moreover, the mapping $\mathcal{T}_{u_*,\varepsilon}$ is defined by 
	$\mathcal{T}_{u_{*},\varepsilon}(v_{0})=v|_{[t_{*},T]\times  \overline{L(3\theta)}}$, i.e. the restriction of the solution $v$ with initial data $v_{0} \in B^{L^{1}(\Omega)}_{\varepsilon}(u_{0})\cap \mathcal{C}$ to the set $[t_{*},T]\times  \overline{L(3\theta)}.$
	Assumption \eqref{eq:HC} and the estimates \eqref{eq:smoothpart} and \eqref{eq:contpart} show that the 
	hypotheses of Proposition  \ref{EfendievZelik2008} hold with the constants $\mu, \kappa$ and $\eta=\frac{1-\delta}{2}<\frac{1}{2}$. Hence, for any $\nu \in (0, \frac{\delta}{2})$ the discrete time semigroup $\{S(n)\}_{n\in\mathbb{N}_0}$ possesses an exponential attractor $\mathcal{M}_\nu^d\subset \mathcal{C}$ and its fractal dimension is bounded by 
	$$ 
	\text{dim}_{f}^{L^{1}(\Omega)}(\mathcal{M}_\nu^d)\leq \log_{1/(1-\delta+2\nu)}\widetilde N_{\nu/\kappa\mu},
	$$
	where 
	$$
	\widetilde N_{\nu/\kappa\mu}=N_{\nu/\kappa\mu}^{Y_{u_{*},\varepsilon}}\left(B^{X_{u_{*},\varepsilon}}_1(0)\right)
	=N_{\nu/\kappa\mu}^{C([t_{*},T]\times \overline{L(7\theta/2)})}\left(B^{C^{\alpha}([t_{*},T]\times \overline{L(3\theta)})}_1(0)\right).
	$$
	As seen in the proof of \cite[Theorem 3.1]{EfendievZelik2008} and \cite[Theorem 2.1]{CzajaEfendiev}, the discrete exponential attractor is the closure in $L^1(\Omega)$ of a countable set in $\mathcal{C}$, $\mathcal{M}_\nu^d=\overline{\widetilde{ \mathcal{M}}_\nu^d}$, where $\widetilde{ \mathcal{M}}_\nu^d=\bigcup_{n\in\mathbb{N}_0}V^n$ with $V^n\in\mathcal{C}.$
	
	Finally, to obtain an exponential attractor for the continuous time semigroup
	we follow the construction in \cite{SonnerPhD}, see also \cite{CzajaSonner}. We set $\mathcal{M}_\nu=\overline{\widetilde{ \mathcal{M}}_\nu}$, where the closure is taken in $L^1(\Omega)$ and 
	$$
	\widetilde{ \mathcal{M}}_\nu=\bigcup_{t\in[0,T]}\mathcal{S}(t)\widetilde{ \mathcal{M}}_\nu^d.
	$$
	Due to the H\"older continuity \eqref{eq:HC}, the fractal dimension of the continuous time exponential attractor is bounded by 
	$$
	\text{dim}_{f}^{L^{1}(\Omega)}(\mathcal{M}_\nu)\leq \frac{1}{\alpha}+ \text{dim}_{f}^{L^{1}(\Omega)}(\mathcal{M}_\nu^d),
	$$
	see \cite[Theorem 3.4]{SonnerPhD}. This provides the estimate for the fractal dimension of the exponential attractor in Theorem \ref{EAS} with 
	$J=[t_*,T],$ $\Omega_1=L(3\theta), \Omega_2=L(7\theta/2)$, $\eta_1=1-\delta+2\nu$ and $\eta_2=\frac{\nu}{2\kappa},$
	which completes the proof.
\end{proof}

The proof of Theorem~\ref{EAS} can be extended for the coupled problem \eqref{eq:CP}. We only provide a sketch of the proof as the main difficulty lies in the degenerate equation.

\begin{proof}[Proof of Theorem~\ref{EAC}:] 
	We indicate the main ideas of the proof. The theorem can be shown by repeating and adjusting the arguments in the proof of Theorem~\ref{EAS} for the coupled system. 
	As for the scalar problem we introduce a set 
	\begin{align*}
		\mathcal{C}:=\overline{\bigcup_{t\geq 1}\mathcal{T}(t)\mathcal{Y}},
	\end{align*}
	where the closure is taken in $C^{\alpha}(\Bar{\Omega})\times C^{\alpha}(\Bar{\Omega})$.
	
	Let $w_{0}=(u_{0},v_{0})\in \mathcal{C}, u_{0}\neq 0$ and $L(\theta)=L(\theta,u_{0}), S(\theta)=S(\theta,u_{0})$ denote the sets defined in \eqref{eq:LS}. Due to \eqref{eq:signcondf2}, we can fix $\theta,\beta>0$ such that
	\begin{align*}
		f_{u}(\cdot,u,v)<-3\beta<0\quad \text{in}\ \Omega\qquad \forall |u|<5\theta,\  v\in [0,1].
	\end{align*} 
	By \eqref{eq:HC2}, the solution $w=(u,v)$ belongs to $C^{\alpha}$ with respect to $t$ and $x$. Then, in the same manner as in Lemma~\ref{sign}, we conclude that for an arbitrary solution $\Tilde{w}:=(\Tilde{u},\Tilde{v})$ of the problem \eqref{eq:CP} with initial data $\Tilde{w}_{0}=(\Tilde{u}_{0},\Tilde{v}_{0})\in B^{L^{1}(\Omega)\times L^{1}(\Omega)}_{\varepsilon}(w_{0})\cap \mathcal{C}$, the following estimates hold, uniformly with respect to $\Tilde{w}_{0}$,
	\begin{align}
		\begin{split}
			f_{u}(x,\Tilde{u}(t,x),\Tilde{v}(t,x))&<-\beta\qquad \text{for}\ x\in S(5\theta), t\in [0,T],\\
			|\Tilde{u}(t,x)|&>\theta/4\qquad \text{for}\ x\in L(\theta),\ t\in [0,T].
		\end{split}
		\label{eq:ressigncond4}
	\end{align}
	
	Now, to obtain the smoothing property for the difference of two solutions of \eqref{eq:CP}, let $w_{1}=(u_{1},v_{1})$ and $w_{2}=(u_{2},v_{2})$ be two solutions of the coupled problem \eqref{eq:CP} with initial data $w_{i}(0)\in B^{L^{1}(\Omega)\times L^{1}(\Omega)}_{\varepsilon}(w_{0})\cap \mathcal{C}$, $i=1,2$. Define $\xi=(\xi_{1},\xi_{2}):=w_{1}-w_{2}$, which solves the system
	\begin{align}
		\begin{split}
			\partial_{t}\xi_{1}&=\Delta(l_{1}(t)\xi_{1})+l_{2}(t)\xi_{1}+f(\cdot,u_{2},v_{1})-f(\cdot,u_{2},v_{2}),\\
			\partial_{t}\xi_{2}&=\Delta \xi_{2}+g(\cdot,u_{1},v_{1})-g(\cdot,u_{2},v_{2}),\\
			\xi|_{t=0}&=w_{1}(0)-w_{2}(0) 
		\end{split}
		\label{eq:differenceprob2}
	\end{align}
	where
	\begin{align}
		\begin{split}
			l_{1}(t)&:=\int_{0}^{1}\phi'(su_{1}(t)+(1-s)u_{2}(t))ds,\\
			l_{2}(t)&:=\int_{0}^{1}f_{u}(\cdot, su_{1}(t)+(1-s)u_{2}(t),v_{1})ds.
		\end{split}
		\label{eq:l1-l2def2}
	\end{align} 
	By following the arguments in Lemma~\ref{smoothinglemma} and using the boundedness of the functions $f$ and $g$, we conclude that for every $t_{0}\in(0,T)$, we have
	\begin{equation}
		\lVert \xi_{1}\rVert_{C^{\alpha}([t_{0},T]\times \overline{L(3\theta)})}\leq C_{t_{0}}\lVert \xi_{1}(0)\rVert_{L^{1}(\Omega)},\label{eq:smoothingu}
	\end{equation}
	for some constant $C_{t_0}>0.$
	Moreover, classical interior regularity estimates for semilinear uniformly parabolic equations (see \cite{Ladyzhenskaya}) imply that
	\begin{equation}
		\lVert \xi_{2}\rVert_{C^{\alpha}([t_{0},T]\times \overline{\Omega})}
		\leq C_{t_{0}}\lVert \xi_{2}(0)\rVert_{L^{1}(\Omega)}.\label{eq:smoothingv}
	\end{equation}
	Hence, we obtain, by using \eqref{eq:smoothingu}-\eqref{eq:smoothingv}, for arbitrary $t_0\in(0,T)$
	\begin{equation}
		\lVert \xi\rVert_{C^{\alpha}([t_{0},T]\times \overline{L(3\theta)})\times C^{\alpha}([t_{0},T]\times \overline{\Omega})}\leq C_{t_{0}}\lVert \xi(0)\rVert_{L^{1}(\Omega)\times L^{1}(\Omega)}.\label{eq:smoothing2}
	\end{equation}
	where the constant $C_{t_{0}}>0$ is independent of  $\varepsilon\leq \varepsilon_0, w_{0}$ and the trajectories $w_{1}$ and $w_{2}$.
	
	In order to establish the contraction property for the difference of two solutions of \eqref{eq:differenceprob2}, we multiply the first equation by\begin{equation*}
		\psi(x)\text{sgn}(\xi_{1}(t,x))=\psi(x)\text{sgn}(\psi(x)l_{1}(t,x)\xi_{1}(t,x)),
	\end{equation*}
	where $\psi$ is the cut-off function defined in \eqref{eq:cutoffdef}, and integrate over $\Omega$. Then, by $\eqref{eq:ressigncond4}_{1}$, \eqref{eq:FG1} and \eqref{eq:cutoffbdd}, we obtain
	\begin{align*}    
		\partial_{t}\left<\psi,|\xi_{1}(t)|\right>+\beta\left<\psi,|\xi_{1}(t)|\right>
		&\leq \int\limits_{L(4\theta)}\Delta\psi(x)l_{1}(t,x)|\xi_{1}(t,x)|dx\\
		&\quad+\int\limits_{\Omega}|f(x,u_{2}(t,x),v_{1}(t,x))-f(x,u_{2}(t,x),v_{2}(t,x))||\psi(x)|dx\\&\leq C(||\xi_{1}(t)||_{L^{1}(L(4\theta))}+||\xi_{2}(t)||_{L^{1}(\Omega)}),
	\end{align*}
	which, by Gronwall's lemma and \eqref{eq:L1C}, yields
	\begin{align}
		\lVert \xi_{1}(T)\rVert_{L^{1}(S(4\theta))}\leq \frac{1-\delta}{2}\lVert\xi_{1}(0)\rVert_{L^{1}(\Omega)}+C_{t_*}\left(\lVert\xi_{1}\rVert_{L^{1}((t_{*},T)\times L(4\theta))}+\lVert\xi_{2}\rVert_{L^{1}((t_{*},T)\times \Omega)}\right), \label{eq:contractu}
	\end{align}
	for a sufficiently small $t_{*}\in(0,T)$ and $\delta<1$. Therefore, together with the  $L^{1}$-contraction estimate for solutions of the semlinear equation
	\begin{align*}
		\lVert \xi_{2}(T)\rVert_{L^{1}(\Omega)}\leq e^{LT}\lVert \xi_{2}(0)\rVert_{L^{1}(\Omega)},
	\end{align*}
	we get the required contraction property, 
	\begin{align}
		\lVert \xi(T)\rVert_{L^{1}(S(4\theta))\times L^{1}(\Omega)}&\leq \frac{1-\delta}{2}\lVert\xi(0)\rVert_{L^{1}(\Omega)\times L^{1}(\Omega)}+C_{t_{*}}\lVert\xi\rVert_{L^{1}((t_{*},T)\times L(4\theta))\times L^{1}((t_{*},T)\times \Omega)}, \label{eq:contract2}
	\end{align}
	for some constant $C_{t_{*}}>0$ and sufficiently small $t_{*}\in(0,T)$.
	
	Finally, using that
	\begin{align*}
		\lVert \xi(T)\rVert_{L^{1}(\Omega)\times L^{1}(\Omega)}&\leq \lVert \xi(T)\rVert_{L^{1}(S(4\theta))\times L^{1}(\Omega)}+\lVert \xi(T)\rVert_{L^{1}(L(7\theta/2))\times L^{1}(\Omega)},
		\\
		\lVert \xi(T)\rVert_{L^{1}(L(7\theta/2))\times L^{1}(\Omega)}&\leq C\lVert \xi\rVert_{C([t_{*},T]\times\overline{L(7\theta/2)})\times C([t_{*},T]\times \Bar{\Omega})},
	\end{align*}
	together with the estimates \eqref{eq:smoothing2} for $t_0=t_*$ and \eqref{eq:contract2}, we conclude that
	\begin{align*}
		\lVert w_{1}-w_{2}\rVert_{C^{\alpha}([t_{*},T]\times \overline{L(3\theta)})\times C^{\alpha}([t_{*},T]\times \overline{\Omega})}&\leq \kappa\lVert w_{1}(0)-w_{2}(0)\rVert_{L^{1}(\Omega)\times L^{1}(\Omega)},\\
		\lVert w_{1}(T)-w_{2}(T)\Vert_{L^{1}(\Omega)\times L^{1}(\Omega)}&\leq \frac{1-\delta}{2}\lVert w_{1}(0)-w_{2}(0)\rVert_{L^{1}(\Omega)\times L^{1}(\Omega)}\\&\quad+\mu\lVert w_{1}-w_{2}\rVert_{C([t_{*},T]\times \overline{L(7\theta/2)})\times C([t_{*},T]\times \Bar{\Omega})}.\notag
	\end{align*}
	for all solutions $w_1,w_2$ with initial data $w_{i}(0)\in B^{L^{1}(\Omega)\times L^{1}(\Omega)}_{\varepsilon}(w_{0})\cap \mathcal{C}$, $i=1,2,$ where  $0<\delta<1$, $T>0$, $\kappa=C_{t_{*}}$ and $\mu$ are independent of $\varepsilon\leq \varepsilon_{0}$, $w_{0}$ and $w_1,w_2$. 
	Hence, we can apply Proposition  \ref{EfendievZelik2008} to the discrete time semigroup 
	$\Pi(n)=\mathcal{T}(Tn), n\in\mathbb{N}_{0}$, with $E=L^{1}(\Omega)\times L^{1}(\Omega),$ $K=\mathcal{C}, u_{*}=w_{0},$ $X_{u_{*},\varepsilon}=C^{\alpha}([t_{*},T]\times \overline{L(3\theta)})\times C^{\alpha}([t_{*},T]\times \Bar{\Omega}))$ and $Y_{u_{*},\varepsilon}=C([t_{*},T]\times\overline{L(7\theta/2)})\times C([t_{*},T]\times \Bar{\Omega})$. Moreover, the mapping $\mathcal{T}_{u_{*},\varepsilon}$ is defined by 
	$\mathcal{T}_{u_{*},\varepsilon}(\tilde{w}_{0})=(\tilde{u}|_{[t_{*},T]\times  \overline{L(3\theta)}}, \tilde{v}|_{[t_{*},T]\times  \overline{\Omega}})$, i.e. the restriction of the solution $\Bar{w}=(\Bar{u},\Bar{v})$ with initial data  
	$\Bar{w}_{0}\in B^{L^{1}(\Omega)\times L^{1}(\Omega)}_{\varepsilon}(w_{0})\cap \mathcal{C}$
	to the sets $[t_{*},T]\times  \overline{L(3\theta)}$ and $[t_{*},T]\times  \Omega$, respectively.
	Assumption \eqref{eq:HC2} and the estimates \eqref{eq:smoothpart} and \eqref{eq:contpart} show that the 
	hypotheses of Proposition  \ref{EfendievZelik2008} hold with $\mu, \kappa$ and $\eta=\frac{1-\delta}{2}<\frac{1}{2}$. Hence, for any $\nu \in (0, \frac{\delta}{2})$ the discrete time semigroup $\{\Pi(n)\}_{n\in\mathbb{N}_0}$ possesses an exponential attractor $\mathcal{M}_\nu^d\subset \mathcal{C}$ and its fractal dimension is bounded by  
	$$
	\text{dim}_{f}^{L^{1}(\Omega)\times L^{1}(\Omega)}(\mathcal{M}_\nu^d)\leq \log_{1/(1-\delta+2\nu)}\widetilde N_{\nu/\kappa\mu},\quad \text{where}\ \widetilde N_{\nu/\kappa\mu}=\ N_{\nu/\kappa\mu}^{Y_{u_{*},\varepsilon}}(B^{X_{u_{*},\varepsilon}}_{1}(0)).
	$$
	
	The exponential attractor for the continuous time semigroup is obtained as in the proof of Theorem \ref{EAS}
	which, together with \eqref{eq:HC2}, yields the estimate for the fractal dimension of the exponential attractor in Theorem \ref{EAC} with $J=[t_{*},T], \Omega_{1}=L(3\theta), \Omega_{2}=L(7\theta/2)$, $\chi_{1}=1-\delta+2\nu$ and $\chi_{2}=\frac{\nu}{\kappa\mu}$.
\end{proof}

\section*{Acknowledgement}
 
The first author would like to express sincere gratitude to the Scientific and Technological Research Council of Türkiye (TÜBİTAK) for their generous support and funding for this research under 2219 - International Postdoctoral Research Fellowship Program. Their contribution has been invaluable in enabling the completion of this study.

The authors would like to thank Rados\l aw Czaja for his valuable comments and clarification of inequality \eqref{eq:interpol} and the referee for their careful reading of the manuscript and valuable remarks.

\end{document}